\documentclass[english]{smfart}
\usepackage{amssymb} 
\usepackage{amsmath}
\usepackage{amsthm} 
\usepackage{mathrsfs}
\usepackage{fontenc}
\usepackage{graphicx} 
\usepackage{color} 
\usepackage{manfnt}
\usepackage{enumitem} 
\usepackage[utf8]{inputenc} 
\usepackage[T1]{fontenc}
\usepackage{geometry}
\usepackage[all]{xy} 
\usepackage{hyperref} 
\usepackage{vmargin}
\usepackage{pdfsync}
\usepackage{upgreek}
\usepackage{datetime}
\usepackage{todo}
\newtheorem{thm}{Theorem}[section]

\newtheorem{theorem}[thm]{Theorem}

\newtheorem{lem}[thm]{Lemma}

\newtheorem{prop}[thm]{Proposition}

\newtheorem{propnots}[thm]{Proposition and notation}
\theoremstyle{definition} 
\newtheorem{rem}[thm]{Remark}

\newtheorem{defn}[thm]{Definition} 
\newtheorem{definition}[thm]{Definition} 
\newtheorem{defn-lem}[thm]{Definition-Lemma}

\newtheorem{example}[thm]{Example}

\newtheorem{notation}[thm]{Notation}
\newtheorem{notations}[thm]{Notations}

\theoremstyle{remark}
 
\numberwithin{equation}{section}

\newcommand{\A}{\mathbb A}
\newcommand{\B}{\mathscr B}
\newcommand{\C}{C} 

\newcommand{\E}{\mathcal E} 
\newcommand{\F}{\overline{F}}
\newcommand{\I}{\mathcal I}
\newcommand{\K}{\mathbf K}
\renewcommand{\L}{\mathscr L} 
\newcommand{\M}{\mathcal M}
\newcommand{\N}{\mathcal N}
\renewcommand{\P}{\mathbb P}
\newcommand{\R}{\mathrm{R}} 
\newcommand{\X}{\overline{X}}
\newcommand{\Z}{\mathbb Z} 

\renewcommand{\AA}{\mathbb{A}}
\newcommand{\CC}{\mathbb C} 
\newcommand{\LL}{\mathbb L}
\newcommand{\NN}{\mathbb N}
\newcommand{\PP}{\mathbb{P}}
 
\newcommand{\RR}{\mathbb R}
\newcommand{\ZZ}{\mathbb Z}

\newcommand{\bc}{\mathbf{c}}

\newcommand{\mI}{\mathcal{I}}
\newcommand{\mP}{\mathcal{P}}
\newcommand{\mS}{\mathcal{S}}

\newcommand{\f}{\overline{f}}

\renewcommand{\i}{\mathrm{i}}
\renewcommand{\k}{\mathbf k}
\newcommand{\m}{\mathfrak m}

\newcommand{\unit}{\mathrm{u}}
\newcommand{\x}{\mathrm{x}}

\newcommand{\abs}[1]{\left\vert#1\right\vert}
\newcommand{\ac}[1]{\overline{\text{ac}}\!\left( #1 \right)}
\newcommand{\Ac}[1]{\overline{\text{ac}}\:#1}

\newcommand{\ord}[1]{\text{ord}\!\left( #1 \right)} 
\newcommand{\Ord}[1]{\text{ord}\:#1}

\newcommand{\ds}{\displaystyle}
\newcommand{\eps}{\varepsilon} 
\newcommand{\Eu}{\chi_c}
\newcommand{\fdelta}{\updelta}
\newcommand{\gen}{\mathrm{gen}}

\newcommand{\Gm}{\mathbb G_{m}} 
\renewcommand{\line}{\mathrm{l}}
\newcommand{\mes}{\text{mes}}
\newcommand{\mgg}{\mathcal M_{\mathbb G_{m}}^{\mathbb G_{m}}} 
\newcommand{\MHM}{\mathrm{MHM}}
\newcommand{\mon}{\mathrm{mon}}

\newcommand{\mot}{\mathrm{mot}}
\newcommand{\motive}{\mathrm{M}}

\newcommand{\nMilnor}{\upmu}

\newcommand{\Newton}{\mathrm{Newton}}
\newcommand{\rs}{\mathrm{rs}}
\newcommand{\ra}{\rightarrow}
\newcommand{\racineunite}{\boldsymbol{\mu}}

\newcommand{\setI}{\mathrm{I}}
\newcommand{\Spec}{\text{Spec}\:} 
\newcommand{\suml}{\sum\limits}
\newcommand{\Supp}{\text{Supp}}
 
\renewcommand{\top}{\mathrm{top}}
\newcommand{\tsigma}{\boldsymbol{\upsigma}}
\newcommand{\Var}{\text{Var}}
\newcommand{\vertex}{\mathrm{v}}
\newcommand{\vertexw}{\mathrm{w}}

\newcommand\mg{{\mu}}

\definecolor{vert}{rgb}{0,0.6,0.2}

\DeclareRobustCommand{\@pointrait}{.}

\title[Local motivic invariants of rational functions in two variables]{Local motivic invariants of rational functions in two variables}
\author{Pierrette Cassou-Noguès}
\address{Institut de Math\'ematiques de Bordeaux, UMR 5251, Universit\'e de Bordeaux, 350, Cours de la Lib\'eration, 33405, Talence Cedex, France}
\email{Pierrette.Cassou-Nogues@math.u-bordeaux.fr}
\author{Michel Raibaut} 
\address{Laboratoire de Math\'ematiques, UMR 5127, Universit\'e Savoie Mont-Blanc, B\^atiment Chablais, Campus Scientifique, 
Le Bourget du Lac, 73376 Cedex, France} 
\email{Michel.Raibaut@univ-smb.fr}
\urladdr{https://raibautm.perso.math.cnrs.fr/}
\begin{document} 
\begin{abstract}
	Let $P$ and $Q$ be two polynomials in two variables with coefficients in an algebraic closed field of characteristic zero. We consider the rational function $f=P/Q$. For an indeterminacy point $\x$ of $f$ and a value $c$, we compute the motivic Milnor fiber $S_{f,\x, c}$ in terms of some motives associated to the faces of the Newton polygons appearing in the Newton algorithms of $P-cQ$ and $Q$ at $\x$, without any condition of non-degeneracy or convenience. In the complex setting, assuming for any $(a,b)\in \CC^2$ that $\x$ is a smooth or an isolated critical point of $aP+bQ$, and the curves $P=0$ and $Q=0$ do not have common irreducible component, we prove that the topological bifurcation set $\B_{f,\x}^{\top}$ is equal to the motivic bifurcation set $\B_{f,\x}^{\mot}$ and they are computed from the Newton algorithm.
\end{abstract}
\maketitle{} 
\date{\today,\:\currenttime}
\section*{Introduction} 

Let $P$ and $Q$ be two polynomials in $\CC[x_{1},\dots,x_{d}]$ with $d \geq 1$. We denote by $I$ the common zero set of $P$ and $Q$ and by $f$ the rational function $P/Q$ well defined on $\AA_{\CC}^{d} \setminus I$ to $\PP^{1}_{\CC}$. 
For an indeterminacy point $\x \in I$ and a value $c$ in $\PP^{1}_{\CC}$, Gusein-Zade, Luengo and Melle-Hern\'andez in \cite{GusLueMel98a,GusLueMel01a}, proved that $f$ is a Milnor $C^{\infty}$ locally trivial fibration on some $B(\x,\eps)\setminus I$ for $\eps$ small enough, over a punctured neighborhood of $c$. 
They defined a \emph{Milnor fiber of $f$ at $\x$ for the value $c$}, denoted by $F_{\x,c}$, endowed with a monodromy action $T_{\x,c}$ induced by the fibration. 
If the fibration can be extended as a trivial fibration over a neighborhood of $c$, then $c$ is called \emph{typical value}, and \emph{atypical value} otherwise. They proved that the set of atypical values is finite and called it, the \emph{topological bifurcation set} $\B_{f,\x}^{\top}$ of the germ of $f$ at $\x$. 
Later, several authors (for instance \cite{SieTiba,Tiba}, \cite{BodPic07a,BodPicSea09a}, \cite{GL}, \cite{VZG}, \cite{Nguyen}, \cite{NST}) studied singularities of meromorphic functions, locally or at infinity. This is related to the study of pencils $aP+bQ=0$ (for instance \cite{TraWeb97a}, \cite{Par99a}, \cite{SieTiba}, \cite{Delgado-Maugendre}).

Working over a characteristic zero field $\k$, with $P$ and $Q$ polynomials with coefficients in $\k$, using the motivic integration theory, introduced by Kontsevich in \cite{Kon95a}, and more precisely constructions of Denef--Loeser in \cite{DenLoe98b, DenLoe99a,DenLoe02a} and  Guibert--Loeser--Merle in \cite{GuiLoeMer06a}, the second author in \cite{Raibaut-fractions} and Nguyen--Takeuchi in \cite{Nguyen-Takeuchi} (see also \cite{GVLM} and \cite{VZG})
defined a \emph{motivic Milnor fiber} of $f$ at $\x$ and a value $c$, denoted by $S_{f,\x,c}$ (section \ref{section-Sfxc}). It is an element of 
$\mathcal M_{\Gm}^{\Gm}$, a modified Grothendieck ring of varieties over $\k$ endowed with an action of the multiplicative group $\Gm$ of $\k$. 
When $\k$ is the field of complex numbers, it follows from Denef--Loeser results that the motive $S_{f,\x,c}$ is a ``motivic'' incarnation of the topological Milnor fiber $F_{\x,c}$ endowed with its monodromy action $T_{\x,c}$. 
For instance, the motive $S_{f,\x,c}$ realizes on the Euler characteristic of $F_{\x,c}$ and the classes in appropriate Grothendieck rings of the associated mixed Hodge modules and limit mixed Hodge structures (Theorem \ref{realisationMHMlocal}).
In \cite{Raibaut-fractions}, the second author proved that the set of values $c$ such that $S_{f,\x,c}\neq 0$ is a finite set, 
denoted by $\B_{f,\x}^{\mot}$ and called \emph{motivic bifurcation set} (Definition \ref{Bfmot}).

In this article, similarly to \cite{carai-antonio, carai-infini}, assuming $\k$ algebraically closed, we investigate the case $d=2$ in full generality (namely without any assumptions of convenience or non degeneracy w.r.t any Newton polygon) using ideas of Guibert in \cite{Gui02a}, Guibert--Loeser--Merle in \cite{GuiLoeMer05a, GuiLoeMer09b}, and the works of the first author and Veys in the case of an ideal of $\k[[x,y]]$ in \cite{CassouVeys13, Cassou-Nogues-Veys-15}. 

More precisely, in Theorem \ref{thmSf00c}, we give an inductive expression of the motive $S_{f,\x,c}$, in terms of some motives associated to faces of the Newton polygons appearing in the Newton algorithm of $P-cQ$ and $Q$ recalled in section \ref{section1}. 
Furthermore, in Theorem \ref{thmBfNewtonBftop} and Theorem \ref{thmegaliteBf}, in the smooth case or isolated singularity case, assuming that the curves $P=0$ and $Q=0$ do not have a common irreducible component, we have the equality 
\(\B_{f,\x}^{\top}=\B_{f,\x}^{\Newton}=\B_{f,\x}^{\mot}\)
where $\B_{f,\x}^{\Newton}$ is the Newton bifurcation set of $f$ which is defined in an algorithmic way (Definition \ref{Newtongen}) using the Newton algorithm of $P-\bc Q$ and $Q$ with $\bc$ an indeterminate. Up to factorisation of polynomials, this gives in particular an algorithm to compute the usual topological bifurcation set $\B_{f,\x}^{\top}$ without assumptions of convenience or non degeneracy toward any Newton polygon.

\section{Newton algorithm} \label{section1} \label{section:Newton-algorithm-local}
Let $\NN$ be the set of non-negative integers. Let $\k$ be an algebraically closed field of characteristic zero, with $\Gm$ its multiplicative group. Recall first standard notation used throughout this article.

\begin{defn} \label{defDeltaE} \label{def:Newton-polygon-height}
	For any $E \subset \NN^2$, we denote by $\Delta (E)$ the smallest convex set containing
	\[E+\left(\RR_{\geq 0}\right)^2=
		\left\{\begin{array}{c|c} \rm{v}+\rm{w} & \text{v}\in E, \rm{w} \in \left(\RR_{\geq 0}\right)^2 \end{array}\right\}.
	\] 
	A subset $\Delta \subset \left( \RR_{\geq 0} \right)^2$ is called \emph{Newton diagram} if $\Delta = \Delta(E)$ for some subset $E \subset \NN^2$. The smallest set $E_0$ of $\NN^2$ such that 
	$\Delta=\Delta (E_0)$ is called \emph{set of vertices} of $\Delta$, it is a finite set. 

	Let $\Delta\subset \left( \RR_{\geq 0} \right)^2$ be a Newton diagram and $E_0 = \{\vertex_0, \dots, \vertex_d\}$ be its set of vertices with for any $i \in [\![0,d]\!]:= \{0,\dots,d\}$, $\vertex_i = (a_i,b_i) \in \NN^2$ . By convexity, we order $E_0$ such that for any $i \in [\![1,d]\!]$, we have $a_{i-1}<a_i$ and $b_{i-1}>b_i$ and we write $v_{i}>v_{i-1}$. For such $i$, we denote by $S_i$ the segment $[\vertex_{i-1},\vertex_i]$ and by $\line_{S_i}$ the line supporting $S_i$. 
	We define the \emph{Newton polygon} of $\Delta$ as
	\[\N(\Delta)=\{S_i\}_{i \in [\![1,d]\!]} {\cup \{\vertex_i\}_{i \in [\![0,d]\!]}},\] 
	the \emph{height} of $\Delta$ as the integer $h(\Delta)=b _0-b_d$, the \emph{one dimensional faces} of $\N(\Delta)$ as the segments $S_i$, the \emph{zero dimensional faces} of $\N(\Delta)$ as the vertices $\vertex_i$ and among them the \emph{vertical face} 
	$\gamma_v$ as $\vertex_0$ and the \emph{horizontal face} $\gamma_h$ as $\vertex_d$.
\end{defn}

\begin{defn}
	\label{defcoeffsuppport} \label{defNewtondiagram} \label{defDelta'}\label{def:Newton-polygon-origin}
	\label{def:facepolynomial} \label{def:zero-partie-initiale}\label{def:roots}\label{def:smooth}
	Let $P= \sum c_{a,b}x^{a}y^{b}$ in $\k[x,y]$.	
	The \emph{support} of $P$, is by definition the set 
	\[\Supp(P)=\left\{ \begin{array}{c|c} (a,b)\in \NN^2 & c_{a,b} \neq 0 \end{array}\right\}.\]
	Sometimes, we will denote by $c_{P}(a,b)$ the coefficient $c_{a,b}$ of $P$.
	We define the Newton diagrams 
	\[\Delta(P):=\Delta(\Supp(P)), \Delta'(P) := \Delta\left( \Supp(P) \setminus (\NN \times \{0\}) \right).\]	
	the \emph{Newton polygon at the origin} $\N(P):=\N(\Delta(P))$ and $h(P):=h(\Delta(P))$ the height of $P$.

	To any face $\gamma$ of $\N(P)$, we associate a \emph{face polynomial} $P_{\gamma}$ with a set of roots $\R_{\gamma}$:
	\begin{itemize}
		\item If $\dim{\gamma}=0$, then $\gamma$ is a point $(a,b)$, $P_{\gamma}:=c_{(a,b)}x^{a}y^{b}$ and $\R_\gamma = \emptyset$.
		\item If $\dim{\gamma}=1$, then $\gamma$ is supported by a line $\line$ and we define 
			\[P_\gamma:=\sum_{(a,b)\in \gamma}c_{(a,b)}x^{a}y^{b}.\]
			As the field $\k$ is algebraically closed, there exists 
			$c \in \k$, $(a_\gamma,b_\gamma) \in \NN^2$,  
			$(p,q)\in \NN^2$ and coprime, there exists $r \in \NN_{>0}$, 
			$R_{\gamma}=\{\mu_i \in \Gm \mid i \in [\![1;r]\!]\}$ called \emph{roots} and all different with \emph{multiplicities} $(\nu_i) \in \left(\NN_{>0}\right)^r$, such that 
		        \[P_\gamma = cx^{a_\gamma}y^{b_\gamma}\prod _{1\leq i\leq r }(y^p-\mg_ix^q)^{\nu_i}.\]
			Following \cite{Kou76a}, $P$ is said \emph{non degenerate} with respect to its Newton polygon $\N(P)$, if for each one dimensional face $\gamma$ in $\N(P)$, the face polynomial $P_\gamma$ has no critical points on the torus $\Gm^2$, in particular all its roots have multiplicity one.
		\item   A face $\gamma \in \N(P)$ is said to be \emph{smooth} if there are points $\vertex,\vertexw \in \gamma$ such that $c_{\vertex} \neq 0$, $c_{\vertexw}\neq 0$, and one of the following conditions is satisfied:
			\begin{itemize}
				\item $\vertex=(x_{\vertex},1)$ and $\vertexw=(x_{\vertexw},0)$, with $x_{\vertex}<x_{\vertexw}$, $\gamma$ would be called \emph{$y$-smooth},
				\item $\vertexw=(0,y_{\vertexw})$ and $\vertex=(1,y_{\vertex})$,with $y_{\vertex}<y_{\vertexw}$ $\gamma$ would be called \emph{$x$-smooth}.
			\end{itemize}
	\end{itemize}
\end{defn}

\begin{defn}
	Let $\setI$ be a finite set. A \emph{rational polyhedral convex cone} of $\RR^{\abs{\setI}}\setminus \{0\}$
	is a convex part of $\RR^{\abs{\setI}}\setminus \{0\}$ defined by a finite number of linear
	inequalities with integer coefficients of type $\ell_1\leq 0$ or $\ell_2>0$ and stable
	under multiplication by elements of $\RR_{>0}$.
\end{defn}

In the well-known following proposition, we introduce notation used throughout this article.

\begin{propnots} \label{prop:dualfan-Newton-local}
	Let $E \subset \NN^2$, $(p,q) \in \NN^2$ with $\gcd(p,q)=1$ and 
		\[\ell_{(p,q)}\colon(a,b) \in \RR^2 \mapsto ap+bq.\]
	\begin{enumerate}
		\item The minimum of the restriction $\ell_{(p,q)\mid \Delta(E)}$, denoted by $\m(p,q)$, is reached on a face denoted by $\gamma(p,q)$ of $\Delta(E)$.
			Furthermore, the linear map $\ell_{(p,q)}$ is constant on the face $\gamma(p,q)$.
			The function $\m$ defined on $\NN^2$ is called \emph{support function} of $\Delta(E)$.
			If $\Delta(E)=\Delta(P)$ for some polynomial $P$, we say that $\m$ is the support function of $\N(P)$.	
		\item For any face $\gamma$ of $\Delta(E)$, we denote by $C_\gamma$ the interior in its own generated vector space in $\RR^2$, of the positive cone generated by the set $\{(p,q) \in \NN^2 \mid \gamma(p,q) = \gamma\}$. This set is called \emph{dual cone} to the face $\gamma$ and is a relatively open rational polyhedral convex cone of $(\RR_{\geq 0})^{2}$.
	\end{enumerate}
	For a one dimensional face $\gamma$, we denote by $\vec{n}_\gamma$ the normal vector to the face $\gamma$ with integral non negative coordinates and the smallest norm. With these notation we have the following properties:
	\begin{enumerate}[resume] 	
		\item For any one dimensional face $\gamma$ of $\Delta(E)$, $C_\gamma = \RR_{>0} \vec{n}_\gamma$.
		\item Any zero dimensional face $\gamma$ of $\Delta(E)$ is an intersection of two one dimensional faces $\gamma_1$ and $\gamma_2$ of $\Delta(E)$ (may be not compact), and $C_\gamma = \RR_{>0}\vec{n}_{\gamma_1} + \RR_{>0} \vec{n}_{\gamma_2}$.
		\item The set of dual cones $\E_{\Delta}=\{C_\gamma \mid \gamma \in \N(\Delta(E))\}$ is a fan\footnote{We refer for instance to \cite[\S 1.4]{Fulton-toric} for the definition of a fan.} of $(\RR_{\geq 0})^{2}$, called \emph{dual fan} of $\Delta(E)$.
	\end{enumerate}
\end{propnots}

\begin{defn}\label{defEc}
	Let $(P,Q)\in \k[x,y]^2$. We assume $(P,Q)$ not in $\k[x]^2$ or $\k[y]^2$. Let $c \in \k$. 
	We denote by $\E_c$ the fan of $(\RR_{\geq 0})^2$ induced by the Minkowski sum $\N(P-cQ)+\N(Q)$. It is a well-known fact (\cite[Proposition 6.2.13]{Cox}), that this fan is the coarsest common refinement of the dual fans of $\N(P-cQ)$ and $\N(Q)$. For any cone $C$ of $\E_c$, there is a face $\gamma_{P-cQ}(C)$ in $\N(P-cQ)$ and a face $\gamma_Q(C)$ in $\N(Q)$, such that $C$ is the intersection of the dual cones 
	$C_{\gamma(P-cQ)}$ and $C_{\gamma(Q)}$, namely for any $(\alpha,\beta) \in C$, the associated faces to $(\alpha,\beta)$ in 
	$\N(P-cQ)$ and $\N(Q)$ are $\gamma(P-cQ)$ and $\gamma(Q)$. If $\dim{C}=2$ then $\gamma_{P-cQ}(C)$ and $\gamma_Q(C)$ are zero dimensional, if $\dim{C}=1$ then $\gamma_{P-cQ}(C)$ or $\gamma_Q(C)$ is one dimensional.
	We extend these definition to $c=\infty$, considering the rational function $Q/P$ and defining $\E_\infty$ as the fan $\E_0$ for $Q/P$.
\end{defn}

We recall the Newton algorithm and refer to \cite{Wall,ArtCasLue05a, CassouVeys13, Cassou-Nogues-Veys-15, carai-antonio} for more details.

\begin{defn}\label{defn:Newton-map-local} 
	Let $(p,q)\in\NN^2$ and $(p',q')\in\NN^2$ such that $pp'-qq'=1$. Let $\mu\in \Gm$.
	We define the \emph{Newton transformation} associated to $(p,q,\mu)$ as the application
	\begin{equation} \label{formule:Newtontransform} 
		\begin{matrix} 
			&\tsigma_{(p,q,\mu)}&\colon &\k[x,y] &\to& \k[x_1,y_1]\\ 
			&&&P(x,y)&\mapsto &P(\mu^{q'}x_1^p, x_1^q(y_1+\mu ^{p'}))\\ 
		\end{matrix}.
	\end{equation}
	We call $\tsigma_{(p,q,\mu)}(P)$ the \emph{Newton transform} of $P$ and denote it by $P_{\tsigma_{(p,q,\mu)}}$ or simply $P_{\tsigma}$. More generally, let $\Sigma _n=(\tsigma _1,\dots, \tsigma _n)$ be a finite sequence of Newton maps $\tsigma_i$, we define the \emph{composition} $P_{\Sigma _n}$ by induction: 
	$P_{\Sigma _1}=P_{\tsigma_1}$, and for any $i$, $ P_{\Sigma _i}=(P_{\Sigma _{i-1}})_{\tsigma _i}$.
\end{defn}

\begin{lem}[{\cite[Lemma 2]{carai-antonio}}]\label{lem:Newton-alg} 
	Let $(p,q)\in \NN^2$ with $\gcd(p,q)=1$. 
	Let $\mu\in \Gm$.
	Let $P\in \k[x,y]$, $P\neq 0$, ${P_1=\tsigma_{(p,q,\mu)}(P)\in \k[x_1,y_1]}$ be its Newton transform  and $\m$ be the support function of $\N(P)$. 
	\begin{enumerate}
		\item If there does not exist a one dimensional face $\gamma$ of $\mathcal{N}(P)$ whose
			supporting line has equation $pa +qb=N$, for some $N$,
			then there is a polynomial $\unit(x_1,y_1)\in \k[x_1,y_1]$ with $\unit(0,0)\neq 0$ such that
			$P_1(x_1,y_1)=x_1^{\m(p,q)} \unit(x_1,y_1)$.
		\item If there exists a one dimensional face $\gamma$ of $\mathcal{N}(P)$ with supporting line
			of equation {$pa +qb = N$}, if $\mu$ is not a root of $P_\gamma$, then $\m(p,q)=N$ and there is a polynomial $\unit(x_1,y_1) \in \k[x_1,y_1]$ with $\unit(0,0)\neq 0$ such that $P_1(x_1,y_1)=x_1^{N} \unit(x_1,y_1)$.
		\item If there exists a one dimensional face $\gamma$ of $\mathcal{N}(P)$ with supporting line of equation {$pa +qb = N$}, if $\mu$ is a root of $P_\gamma$ of multiplicity $\nu$ then $\m(p,q)=N$ and there is $Q_1(x_1,y_1) \in \k[x_1,y_1]$ with  
		$Q_1(0,0)=0$ and $Q_1(0,y_1)$ of valuation $\nu$, such that $P_1(x_1,y_1)=x_1^{N} Q_1(x_1,y_1).$
		In that case we have in particular the inequality $h(P)\geq \nu \geq h(P_1)$.
	\end{enumerate} 
\end{lem}
\begin{rem}
	If $P_1(x_1,y_1)=x_1^{n_1}y_1^{m_1}\unit(x_1,y_1)$, where
	$(n_1,m_1)\in \NN^2$ and $\unit\in \k[x_1,y_1]$ is a unit in $\k[[x_1,y_1]]$, we say for
	short that $P_1$ is a \emph{monomial times a unit}. From this lemma, we see that there is a finite number of triples $(p,q,\mu)$ such that $\tsigma _{(p,q,\mu)}(P)$ is
	eventually not a monomial times a unit in $\k[[x_1,y_1]]$. These triples are given by the equations of the faces of the Newton polygon and the roots of the
	corresponding face polynomials.
\end{rem}

\begin{defn}\label{def:algo-Newton}
	The \emph{Newton algorithm} of a polynomial $P\in \k[x,y]$ is defined by induction (see for instance \cite[Theorem 1, Lemma 4]{carai-antonio}).
	It starts by applying Newton transformations given by the equations of the faces of the Newton polygon $\N(P)$ and the roots of the corresponding face polynomials. 
	Then, this process is applied on each Newton transform until a \emph{base case} of the form $\unit(x,y)x^{M}y^m$ or 
	$\unit(x,y)x^{M}(y-\mu x^q+R(x,y))^{m}$ is obtained with $\unit(x,y)\in \k[[x,y]]$ with $\unit(0,0)\neq 0$, $\mu\in \Gm$, $(M,m)\in \NN^2$, $q\in \NN$ and $R(x,y)=\sum_{a+bq>q} c_{a,b}x^ay^b \in \k[x,y]$. The output of the algorithm is the tree of the Newton transform polynomials produced (see for instance \cite{CassouVeys13}).
\end{defn}

\section{Motivic setting} \label{sectionmotivicsetting}

We will freely use motivic integration in this article and we refer to \cite{DenLoe01b, Vey06a, Loo02a, CNS} for details.

\subsection{Grothendieck rings} \label{rappelK0} \label{inverse-direct}
We denote by $\Var_{\k}$ the category of \emph{$\k$-varieties} (separated reduced schemes of finite type over $\k$) and by $\M$ the localization of its Grothendieck ring with respect to the class $[\AA^1_{\k}]$. We use the $\Gm$-equivariant variant $\M_{\Gm}^{\Gm}$ (\cite[\S 2]{GuiLoeMer06a} or \cite[\S 2]{GuiLoeMer05a}), generated by isomorphism classes of objects, $X\to \Gm$ endowed with a monomial $\Gm$-action $\sigma$, of the category $\Var_{\Gm}^{\Gm}$. The class of the projection $\AA^1_{\k} \times \Gm \to \Gm$ endowed with the trivial action is denoted by $\LL$. For a variety $(X\to \Gm, \sigma)\in \Var_{\Gm}^{\Gm}$, its fiber in 1 is denoted by $X^{(1)}$ and endowed with an induced action $\sigma^{(1)}$ of the group of roots of unity $\hat{\racineunite}$. The corresponding Grothendieck ring to this operation is denoted by $\M^{\hat{\racineunite}}$ and isomorphic to $\M_{\Gm}^{\Gm}$ (see \cite[Proposition 2.6]{GuiLoeMer06a}).

\subsection{Rational series} 
Let $A$ be one of the rings $\ZZ[\LL, \LL^{-1}]$, 
$\ZZ[\LL, \LL^{-1}, (1/(1-\LL^{-i}))_{i>0}]$ and $\M_{\Gm}^{\Gm}$. 
We denote by $A[[T]]_{\rs}$ the $A$-submodule of $A[[T]]$ generated by 1 and finite products
of terms $p_{e,i}(T)=\LL^{e}T^{i}/(1-\LL^{e}T^{i})$ with $e \in \ZZ$ and $i\in \NN^*$. 
There is a unique $A$-linear morphism
\({\lim_{T \ra \infty}\colon A[[T]]_{\rs}\to A}\)
such that for any subset
$(e_{i},j_{i})_{i\in \setI}$ of $\ZZ \times \NN^*$ with $\setI$ finite $\lim_{T \ra \infty}\prod_{i\in \setI} p_{e_{i},j_{i}}(T)=(-1)^{\abs{\setI}}$. 

We recall the following well known lemma 
(\cite[Lemma 1]{carai-antonio}, \cite[Lemma 2.1]{carai-infini}, \cite[Lemme 2.1.5]{Gui02a} and \cite[\S 2.9]{GuiLoeMer06a}).
Its proof is based on summation of geometric series.

\begin{lem}\label{lemmerationalitedescones} \label{lemmedescones}
	Let $\phi$ and $\eta$ be two $\ZZ$-linear forms defined on $\ZZ^2$.
	Let $C$ be a rational polyhedral convex cone of $\RR^2\setminus\{(0,0)\}$, such that $\phi(C)$ and $\eta(C)$ are contained in 
	$\NN$. We assume that for any $n\geq 1$, the set $C_n:=\phi^{-1}(n)\cap C \cap \ZZ^2$ is finite. 
	We consider the formal series
	\[S_{\phi, \eta, C}(T) = \sum_{n\geq 1} \sum_{(k,l)\in C_n} \LL^{-\eta(k,l)} T^n \in \ZZ\!\left[ \LL, \LL^{-1} \right]\!\![[T]].\]
	\begin{enumerate}
		\item \label{cone1} If $C=\RR_{>0} \vec{\omega}_1 + \RR_{>0} \vec{\omega}_2$, where $\vec{\omega}_1$ and $\vec{\omega}_2$ are two non colinear primitive vectors in $\NN^2$ with $\phi(\vec{\omega}_1)>0$ and $\phi(\vec{\omega}_2)> 0$ then, 
		denoting $\mP =( ]0,1]\vec{\omega}_1 + ]0,1]\vec{\omega}_2 )\cap \ZZ^2$, we have
		\begin{equation} \label{casconedim2}
			S_{\phi,\eta,C}(T)=
			\sum_{(k_0,l_0) \in \mP}
			\frac{\LL^{-\eta(k_0,l_0)}T^{\phi(k_0,l_0)}}
			{(1-\LL^{-\eta(\omega_1)}T^{\phi(\omega_1)})(1-\LL^{-\eta(\omega_2)}T^{\phi(\omega_2)})},\:
			\lim_{T \ra \infty} S_{\phi,\eta,C}(T)=1 = \Eu(C),
		\end{equation}
		 where $\Eu$ is the Euler characteristic with compact support morphism.

	 \item \label{cone2} If $C=\RR_{>0}\vec{\omega}$, where $\vec{\omega}$ is a primitive vector in $\ZZ^2$ with 
		$\phi(\vec{\omega})> 0$, then we have
		\begin{equation} \label{casconedim1} 
			S_{\phi,\eta,C}(T)= \frac{\LL^{-\eta(\omega)}T^{\phi(\vec{\omega})}}{1-\LL^{-\eta(\vec{\omega})}T^{\phi(\vec{\omega})}}
			\:\:\text{and}\:\:\lim_{T \ra \infty} S_{\phi,\eta,C}(T)=-1= \chi_{c}(C).
		\end{equation}
	\item \label{cone3} If $C=\RR_{\geq 0} \vec{\omega}_1+ \RR_{>0} \vec{\omega}_2$
		where $\vec{\omega}_1$ and $\vec{\omega}_2$ are two non colinear primitive vectors in $\NN^2$ with $\phi(\vec{\omega}_1)>0$ and $\phi(\vec{\omega}_2)>0$ then, denoting
		$\mP =( ]0,1]\vec{\omega}_1 + ]0,1]\vec{\omega}_2 )\cap \ZZ^2$, we have
		\begin{equation} \label{casconedim2-bis}
			S_{\phi,\eta,C}(T)=
			\sum_{(k_0,l_0) \in \mathcal P}
			\frac{\LL^{-\eta(k_0,l_0)}T^{\phi(k_0,l_0)}}
			{(1-\LL^{-\eta(\vec{\omega}_1)}T^{\phi(\vec{\omega}_1)})(1-\LL^{-\eta(\vec{\omega}_2)}T^{\phi(\vec{\omega}_2)})} 
			+ \frac{\LL^{-\eta(\vec{\omega}_2)}T^{\phi(\vec{\omega}_2)}}{1-\LL^{-\eta(\vec{\omega}_2)}T^{\phi(\vec{\omega}_2)}}
		\end{equation}
		and $\ds \lim_{T \ra \infty} S_{\phi,\eta,C}(T)=0= \chi_{c}(C)$.
\end{enumerate}
\end{lem}
	
\subsection{Arcs} \label{action}
Let $V$ be a $\k$-variety. For any integer $n$, we denote by $\L_{n}(V)$ the \emph{space of $n$-jets} of $V$. 
This set is a $\k$-scheme of finite type and its $\K$-rational points are morphisms $\Spec(\K[t]/t^{n+1})\to V$, for any extension $\K$ of $\k$. 
There are canonical morphisms $\L_{n+1}(V)\to \L_{n}(V)$. These morphisms are $\AA_{\k}^{d}$-bundles when $V$ is smooth with dimension $d$.  
The \emph{arc space} of $V$, denoted by  $\L(V)$, is the projective limit of this system. 
This set is a $\k$-scheme and we denote by $\pi_{n}\colon\L(V)\to \L_{n}(V)$ the canonical morphism {called \emph{truncation map}} modulo $t^{n+1}$, see \cite{DenLoe99a,Loo02a, CNS}.

For a non zero element $\varphi$ in $\K[[t]]$ or in $\K[t]/t^{n+1}$, we denote by $\ord{\varphi}$ the valuation of $\varphi$ and by $\ac{\varphi}$ 
the coefficient  of $t^{\ord{\varphi}}$ in $\varphi$ called the \emph{angular component} of $\varphi$. By convention $\ac{0}$ is zero. 
The multiplicative group $\Gm$ acts canonically on $\L_{n}(V)$ by $\lambda.\varphi(t):=\varphi(\lambda t).$ 
We consider the application \emph{origin} $\pi_0\colon\varphi\mapsto \varphi(0)=\varphi \mod t$.

Let $F$ be a closed
subscheme of $V$ and $\mI_{F}$ be the sheaf ideal of regular functions on $V$ which
vanish on $F$. We denote by $\ord{F}$ the function which assigns to each arc
$\varphi$ in $\L(V)$, the bound $\inf \ord{g(\varphi)}$ where $g$ runs over $\mI_{F,\varphi(0)}$.

\section{Motivic Milnor fibers of a rational function at an indeterminacy point in \texorpdfstring{$\AA^2_\k$}{}}

\subsection{Setting}  \label{setting-cas-local} \label{notations:Xeps-omega}
We fix notation for this section and the rest of this article. 
\begin{itemize}
	\item As $\k$ is assumed to be algebraically closed, forgetting the scheme structure, we simply write $\AA^{1}_{\k}$ and $\PP^{1}_{\k}$ for their set of $\k$-rational points $\AA^{1}_{\k}(\k)$ and $\PP^{1}_{\k}(\k)$. 
		For any $c\in \k$, we identify $c$ with the point $[c:1] \in \PP^{1}_{\k}$ and we denote by $\infty$ the point $[1:0] \in \PP^1_{\k}$.
	\item Let $P$ and $Q$ be two polynomials in $\k[x,y]$. We assume $(P,Q)$ not in $\k[x]^2$ or $\k[y]^2$. 
		We consider the algebraic variety $I$ defined by the ideal $\I=(P,Q)$ 
		and the function ${f=[P:Q] \colon \AA^2_{\k} \setminus I \to \PP^1_{\k}}$ with graph
		\begin{equation} \label{defX} 
			X = \{((x,y),[u:v])\in (\AA^2_{\k} \setminus I) \times \PP^1_{\k} \mid vP(x,y)=uQ(x,y)\}.
		\end{equation}
		We denote by $\X$ the Zariski closure of $X$ in $\A^2_{\k}\times \P^1_{\k}$ and by $\F$ the closed subset $I \times \P^1_{\k}$.
		The singular locus of $\X$ is contained in $\F$. We have the following commutative diagram 
		\[
			\xymatrix{ 
				\AA_\k^2 \setminus I \ar^-{\i}[r] \ar_{f}[d] & \overline{X} \ar^{\f}[ld] \\ 
			\PP^1_\k} 
		\]
		where $\i$ is the open immersion
		given by $\i(x,y)=\left((x,y),f(x,y)\right)$ and $\f$ is the projection along the last coordinate which extends the application $f$ to $\X$.
\end{itemize}

\subsection{Topological point of view}
We assume $\k = \CC$ and use notation of section \ref{setting-cas-local}.

\begin{theorem}[\cite{GusLueMel98a} and {\cite[Theorem 1.2]{GusLueMel01a}}] 
	Let $\x\in I$. For any value $c \in \PP^{1}_{\CC}$, there is $\eps_{0}>0$ such that for all $\eps \in [\eps_{0},0[$, 
			the function $f \colon B(\x,\eps)\setminus I\to \PP^{1}_{\CC}$ is a smooth locally trivial fibration over a punctured disk $D(c,\delta)\setminus \{c\}$ called \emph{Milnor fibration of $f$ at $\x$ for the value $c$}. 
\end{theorem}
\begin{defn}
	The \emph{topological Milnor fiber of $f$ at $\x$ for the value $c$}  
	\[F_{\x,c}:=\{z \in B(\x,\eps)\setminus I \mid f(z)=c'\}\]
	is any fiber of the above fibration endowed with a \emph{monodromy transformation} $T_{\x,c}\colon F_{\x,c}\ra F_{\x,c}$.
\end{defn}
\begin{defn}
	A value $c \in \PP^{1}_{\mathbb C}$ is said to be \emph{typical} for $f$ at $\x \in I$, if for all $\eps$ small enough, the function $f \colon B(\x,\eps)\setminus I\ra \PP^{1}_{\CC}$ is a smooth trivial fibration over a neighborhood of $c$, otherwise $c$ is said to be \emph{atypical}.
	The \emph{topological bifurcation set $\B_{f,\x}^{\top}$ of $f$ at $\x$} is the set of atypical values.
\end{defn}

\begin{theorem}[\cite{GusLueMel98a}, {\cite[Theorem 1.10, Theorem 1.14]{GusLueMel01a}}] \label{Mfc} \label{criterebfx} 
	Assume the curves $P=0$ and $Q=0$ do not have a common irreducible component. Let $\x\in I$.
	Assume for any $c \in \PP^{1}_{\CC}$, $\x$ is an isolated critical point of $P-cQ$. Then the function $c\mapsto \nMilnor(P-cQ,\x)$ is constant over $\PP^{1}_{\CC} \setminus \B_{f,\x}^{\top}$, and its value is denoted by $\nMilnor(P-c_{\gen} Q,\x)$. Under this assumption,
	we have for any $c\in \PP^{1}_{\CC}$
	\begin{equation} \label{caracteristiqueEulerfibredeMilnor} 
		\chi_c(F_{\x,c}) = \nMilnor(P-c_{\gen} Q,\x)-\nMilnor(P-cQ,\x)
	\end{equation}
	and $c$ belongs to $\B_{f,\x}^{\top}$ if and only if $\chi_c(F_{\x,c}) \neq 0$.
\end{theorem}

\subsection{Motivic point of view} \label{section-Sfxc}
Using constructions of Denef--Loeser \cite{DenLoe99a,DenLoe01b} and Guibert--Loeser--Merle \cite{GuiLoeMer06a}, a motivic Milnor fiber of $f$ at a point $\x \in I$ for a value $c$ is introduced in \cite[\S 4]{Raibaut-fractions} (see also \cite{GVLM} and \cite{Nguyen-Takeuchi}). We give in Theorem \ref{thmSf00c} its computation in dimension $2$ in a full generality.
\begin{notation} In the following, we use notation of section \ref{setting-cas-local} and consider the chart of $\X$:
	\[\X_{v \neq 0} = \{((x,y),z)\in \AA^2_{\k}\times \AA^1_\k \mid P(x,y)=zQ(x,y)\}.\]
	The restriction of $\f$ to the chart $\X_{v\neq 0}$ is the projection $z \colon ((x,y),z)\mapsto z$.
\end{notation}

\begin{defn}\label{defmotzetafct}
	For $\delta >0$, we consider the \emph{motivic zeta function} of $f$ at $\x \in I$ and $c \in \AA^{1}_{\k}$
	\begin{equation} \label{zeta2} 
		\left(Z_{\f,X}^{\delta}(T)\right)_{(\x,c)} = 
		\sum_{n\geq 1} \mes\!\left(X_{n,\x,c}^{\delta}\right) T^n  \in \M_{\Gm}^{\Gm}[\![T]\!]
	\end{equation}
	 for any integer $n\geq 1$, we set
	\begin{equation}\label{Xnc}
		X_{n,\x,c}^{\delta} = 
		\left\{ 
			\varphi \in \L(\X_{v\neq 0}) | \varphi(0)=(\x,c),\: 
			\ord{z(\varphi)-c} = n,\ord{ \F }\!(\varphi)\leq n\delta
		\right\} 
	\end{equation}
	endowed with the standard action of $\Gm$ on arcs and the structural map
	\[ \begin{array}{ccccl} 
			\ac{z-c} & \colon & X_{n,\x,c}^{\delta} & \to & \Gm \\
			&        & \varphi & \mapsto & \ac{z(\varphi)-c}
		\end{array}
	\]
	and $\mes\!\left(X_{n,\x,c}^{\delta}\right)$ is the motivic measure of $X_{n,\x,c}^{\delta}$ defined as 
	\[\mes\!\left(X_{n,\x,c}^{\delta}\right) = [\pi_{m}(X_{n,\x,c}^{\delta})]\LL^{-2m} \in \M_{\Gm}^{\Gm} \] for any $m\geq m_0$ with $m_0$ large enough. Note that if $\X$ is smooth then $m_0=n$ and if $\X$ is singular, then its singular locus is contained in $\F$ and controlled by the condition $\ord{\F}\!(\varphi)\leq n\delta$ and \cite[Lemma 4.1]{DenLoe99a} is used.

	More generally (and similarly to \cite{Veys01, ArtCasLue05a, CassouVeys13, carai-antonio, carai-infini}), we will need to consider motivic zeta functions with differential forms. 
	Let $\nu \in \NN_{\geq 1}$ and a differential form $\omega$ on $\X$, such that the restriction to the open set $X$ is equal to $x^{\nu-1}dx \wedge dy$ and $\nu=1$ if ``$x=0$'' is not included in $I$. 
	In particular its zero locus is contained in the closed subset $\F$ of $\X$.
	We consider the motivic zeta function
	\begin{equation} \label{zeta1} 
		\left(Z_{\f,\omega,X}^{\delta}(T)\right)_{(\x,c)} = 
		\sum_{n\geq 1} \left( \sum_{k \geq 1} \mathbb L^{-(\nu-1)k} \mes\!\left(X_{n,k,\x,c}^{\delta}\right) \right) T^n 
	\end{equation}
	where for any integers $n\geq 1$ and $k\geq 1$, we set
	\begin{equation}\label{Xepsnk}
		X_{n,k,\x,c}^{\delta} = 
		\left\{ 
			\begin{array}{c|c}
				\varphi \in \L(\X_{v\neq 0}) & 
				\begin{array}{l}        
					\varphi(0)=(\x,c),\: 
					\ord{x(\varphi)} = k,\: \ord{\omega(\varphi)} = (\nu-1)k\\
					\ord{z(\varphi)-c} = n,\ \ord{\F}\!(\varphi)\leq n\delta
				\end{array}
			\end{array}
		\right\} 
	\end{equation}
	endowed with the structural map $\ac{z-c}$ to $\Gm$ and the $\Gm$-standard action on arcs. 
\end{defn}

\begin{rem}\label{rem:arc-carte} Some remarks:
	\begin{itemize}
		\item For any arc $\varphi$, we have $\ord{ \F }\!(\varphi)=\min(\Ord{P(x(\varphi),y(\varphi))},\Ord{Q(x(\varphi),y(\varphi))})$.\\
			So for any $\varphi$ in $X_{n,\x,c}^{\delta}$ or $X_{n,k,\x,c}^{\delta}$, we have 
			\[\Ord{(P-cQ)(x(\varphi),y(\varphi))}-\Ord{Q(x(\varphi),y(\varphi))}=n>0\] 
			implying that $\Ord{P(x(\varphi),y(\varphi))} = \Ord{Q(x(\varphi),y(\varphi))} = \ord{ \F }\!(\varphi)$.
		\item Formula (\ref{zeta2}) is a special case when $\nu = 1$ of formula (\ref{zeta1}).
			When $\nu\geq 1$, by assumption the set ``$x=0$'' is included in $I$, then in formula (\ref{zeta1}), the sum over $k$ is finite. 
			Indeed, by assumption there is an integer $M>0$ such that $x^M$ divides $P$ and $Q$, then $k \leq n \delta /M$. 
	\end{itemize}
\end{rem}

The following result (see \cite{Raibaut-fractions}, \cite{Nguyen-Takeuchi}) is a corollary of \cite[Theorem 4.2.1]{DenLoe98b} and \cite[\S 3.7, \S 3.16]{GuiLoeMer06a}.

\begin{theorem} \label{rationalite} \label{realisationMHMlocal} \label{realisationchic}

	For $\delta$ large enough, the motivic zeta function $\left(Z_{\f,X}^{\delta}(T)\right)_{(\x,c)}$ is rational,  
	\[S_{f,\x,c} : = -\lim_{T \to \infty} \left(Z_{\f,X}^{\delta}(T)\right)_{(\x,c)} \in \M_{\Gm}^{\Gm}\]
	does not depend on $\delta$ and is called \emph{motivic Milnor fiber of $f$ at $\x$ for the value $c$}.

	It realizes on the class of the mixed Hodge module $\left[\mathcal F^{\MHM\circlearrowleft}_{\x,c}\right]$ attached to $f$ in \cite{Sai90a}
	\[\tilde{\chi}^{\MHM}\left(S_{f,\x,c}^{(1)}\right)=\left[\mathcal F^{\MHM\circlearrowleft}_{\x,c}\right] \in K_{0}(\MHM^{\mon}_{\x,c}).\]
	In particular, taking the Euler characteristic realization, we have  
	\[\widetilde{\chi_c}\left(S_{f,\x,c}^{(1)}\right) = \chi_{c}(F_{\x,c}).\]
\end{theorem}
\begin{thm}\label{thmSf00c} \label{formuleSf}
	Let $(P,Q) \in \k[x,y]^2$ not in $\k[x]^2$ or $\k[y]^2$, with $P(0,0)=Q(0,0)=0$ \footnote{This theorem is stated with $\x=(0,0)$ because we use Newton polygons. We can always do this by translation.}.

	Let $f=[P:Q]$, $\overline{f}$, $X$ and $c \in \A^{1}_{\k}$ as in section \ref{setting-cas-local}. 

	Let $\omega$ be the differential form $x^{\nu-1}dx \wedge dy$ with $\nu \geq 1$ as in Definition~\ref{defmotzetafct}.
	
	There is $\delta_0>0$ such that for any $\delta > \delta_0$, $\left(Z_{\f,\omega,X}^{\delta}(T)\right)_{((0,0),c)}$ is rational,
	we denote 
	\[	
		S_{f,(0,0),c}  =  \left(S_{\f, \omega, X}\right)_{((0,0),c)}:= 
		-\ds\lim_{T \to \infty} \left(Z_{\f,\omega,X}^{\delta}(T)\right)_{((0,0),c)} 
	\]
	and we have  
	\begin{multline} \label{formulethmprincipal}
		S_{f,(0,0),c}   
		=  \eps_{C_h}\left[x^{(\gamma_{h}(P-cQ)-\gamma_{h}(Q) \mid (1,0))}y^{(\gamma_{h}(P-cQ)-\gamma_{h}(Q) \mid (0,1))} \colon \Gm^{r_h} \to \Gm, \sigma_{\Gm^{r_h}}\right] 
		\\ + \eps_{C_v}\left[x^{(\gamma_{v}(P-cQ)-\gamma_{v}(Q) \mid (1,0))}y^{(\gamma_{v}(P-cQ)-\gamma_{v}(Q) \mid (0,1))} \colon \Gm^{r_v} \to \Gm, \sigma_{\Gm^{r_v}}\right] \\ 
		- \suml_{C \in \E_c \setminus \{C_{h},C_{v}\},\:\dim C=2} 
		\eps_{C} \left[\frac{(P-cQ)_{\gamma_{P-cQ}(C)}}{Q_{\gamma_{Q}(C)}}\colon\Gm^2 \to \Gm, \sigma_{C}\right] \\
		-  \suml_{C \in \E_c,\ \dim C=1} \eps_{C}
		\left[\frac{(P-cQ)_{\gamma_{P-cQ}(C)}}{Q_{\gamma_{Q}(C)}}\colon\Gm^2 \setminus \{(P-cQ)_{\gamma_{P-cQ}(C)}\cdot Q_{\gamma_{Q}(C)}=0\} \to \Gm, \sigma_{C}\right]\\ 
		+   \suml_{C \in \E_c,\ \dim C=1}\:\: \suml_{\mu \in \R_{\gamma_{P-cQ}}(C)\cup \R_{\gamma_{Q}(C)}}  
		S_{f_{\tsigma_{(p,q,\mu)}},(0,0),c}
	\end{multline}
	where $\gamma_{h}$ and $\gamma_{v}$ are introduced in Definition \ref{def:Newton-polygon-height}, 
	$f_{\tsigma(p,q,\mu)} = [P_{\tsigma(p,q,\mu)}:Q_{\tsigma(p,q,\mu)}]$ and\\ 
	\begin{itemize}
		\item $C_h$ (resp $C_v$) is the cone of the fan $\E_c$ associated to the faces $\gamma_{h}(P-cQ)$ and $\gamma_{h}(Q)$ (resp $\gamma_{v}(P-cQ)$ and $\gamma_{v}(Q)$), $(r_{h},r_{v})\in \{1,2\}^2$ and $(\eps_{C_h},\eps_{C_v}) \in \{-1,0,1\}^2$ in \emph{Proposition~\ref{conedim2}},\\

		\item if $C$ is a two-dimensional cone of $\E_c \setminus \{C_h, C_v\}$ then $\eps_{C} \in \{0,1\}$ in \emph{Proposition~\ref{conedim2}},\\

		\item if $C$ is a one-dimensional cone of $\E_c$ then $\eps_{C} \in \{-1,0\}$ in \emph{Proposition~\ref{conedim1}},\\

		\item  $\sigma_{\Gm}$ and $\sigma_{\Gm^2}$ are the actions of $\Gm$ on $\Gm$ and $\Gm^2$ defined by 
			$\sigma_{\Gm}(\lambda,x)=\lambda x$ and $\sigma_{\Gm^2}(\lambda,(x,y))=(\lambda x, \lambda y)$, and 
			$\sigma_C$ is the action of $\Gm$ induced by a cone $C$ in \emph{Notations~\ref{notation-identification-action}}.
	\end{itemize}
\end{thm} 
\begin{proof}
	We present here the strategy of the proof and refer to section \ref{sectionpreuvethm} for details.
	The rationality result and the independence of the motive $\left(S_{\f, \omega, X}\right)_{((0,0),c)}$ from the differential form $\omega$ (then its equality with $S_{f,(0,0),c}$) are standard and follow from \cite{DenLoe98b, GuiLoeMer06a} (see for instance  \cite[\S 2.5]{carai-infini}).
	Using the Newton algorithm and the strategy of the first author and Veys in \cite{CassouVeys13} (see also \cite{carai-antonio}, \cite{Gui02a} \cite{ArtCasLue05a}), the motivic Milnor fiber $S_{f,(0,0),c}$ is given in terms of Newton polygons of $P-cQ$ and $Q$ and their Newton transforms. Classifying the arcs $(x(t),y(t))$ occurring in the definition of the motivic zeta function by the condition $(\ord{x(t)}, \ord{y(t)}) \in C$ with $C$ a cone of the fan $\E_c$, we decompose (section \ref{sec:decompositionzeta}, formula \ref{zetadecomposition}) the motivic zeta function as 
	\[\left(Z_{\f,\omega,X}^{\delta}(T)\right)_{((0,0),c)} = \sum_{\C \in \E_c} Z_{f-c,\omega,C}^{\delta}(T).\]
	Now, for each cone $C$ in $\E_c$, we classify the arcs $(x(t),y(t))$ that occur by the conditions
	\[(P-cQ)_{\gamma_{P-cQ}(C)}(\ac x, \ac y) \blacktriangle 0 \text{ and } Q_{\gamma_{Q}(C)}(\ac x, \ac y) \blacktriangledown 0\]
	with $(\blacktriangle,\blacktriangledown)\in \{=,\neq\}^2$, which gives the decomposition (formula \ref{decompositiontriangles}) 
	\[Z_{f-c,\omega,C}^{\delta}(T)=\sum_{(\blacktriangle,\blacktriangledown)\in \{=,\neq\}^2}Z_{f-c,\omega,C}^{\delta,(\blacktriangle,\blacktriangledown)}(T).\]
	For each $(\blacktriangle,\blacktriangledown)\in \{=,\neq\}^2$, the rationality and limit of the zeta function 
	$Z_{f-c,\omega,C}^{\delta,(\blacktriangle,\blacktriangledown)}(T)$ is done in the different Propositions \ref{conedim2}, \ref{conedim1}, 
	\ref{prop:rationalityZeqneq} and \ref{zfautrescas}. The well-known idea in the case $(\neq,\neq)$ is the following. Using the definition of the motivic measure, we have
		\[Z_{f-c,\omega,C}^{\delta,(\neq,\neq)}(T) = \motive S_{\C^{\delta}_{(\neq, \neq)}}(T)\]
		for some motive $\motive  \in \M_{\Gm}^{\Gm}$, and $S_{\C^{\delta}_{(\neq, \neq)}}(T)$ a formal series as in Lemma \ref{lemmedescones}
\[
	S_{\C^{\delta}_{(\neq, \neq)}}(T) = \sum_{n\geq 1} \sum_{(\alpha,\beta) \in \C_{(\neq,\neq),n}^{\delta}}\LL^{-\eta(\alpha,\beta)}T^{n}.
\]
Then using Lemma \ref{lemmedescones}, the formal series $S_{\C^{\delta}_{(\neq, \neq)}}(T)$ is rational and 
$\lim\limits_{T \to \infty} S_{\C^{\delta}_{(\neq, \neq)}}(T)=\chi_{c}\big(\C^{\delta}_{(\neq, \neq)}\big)$ is the coefficient $\eps_{C}$ which does not depend on $\delta \geq \delta_0$ for $\delta_0$ large enough. 
The other cases $(=,\neq)$ and $(\neq,=)$ and $(=,=)$ are done using Newton transformations (Propositions \ref{prop:rationalityZeqneq} and \ref{zfautrescas}).
\end{proof}

\begin{rem}\label{remfondamentale} Some remarks:
	\begin{itemize}
		\item See Examples \ref{exemple}, \ref{casdebase1} and \ref{casdebase2}, for computations of the motive $S_{f,(0,0),c}$.
		\item In the case $Q=x^M$ with $M\in \NN_{>0}$ and $c=0$, formula (\ref{formulethmprincipal}) is the computation of the motivic 
			Milnor fiber of $x^{-M}P(x,y)$ at $(0,0)$ in \cite[Theorem 2.22]{carai-infini}.
		\item In a similar way than \cite[section 5]{carai-antonio} and \cite[Corollary 3.12, Corollary 3.25]{carai-infini}, 
			we deduce from formula (\ref{formulethmprincipal}) of Theorem \ref{thmSf00c}, a Kouchnirenko type formula of $\chi_{c}\left(F_{\x,c} \right)$. This formula uses \cite[Proposition 2.4]{carai-infini} and with a similar proof, the following result.
			Let $(p,q) \in (\ZZ_{>0})^2$ coprime, and two quasi-homogeneous polynomials in $\k[x,y]$ 
			\[
				P_1(x,y)=x^{a_1}y^{b_1}\prod_{i \in \setI_1}\left(y^{p}-\mu_i^{(1)}x^q\right)^{\nu_i^{(1)}},\:
				P_2(x,y)=x^{a_2}y^{b_2}\prod_{i \in \setI_2}\left(y^{p}-\mu_i^{(2)}x^q\right)^{\nu_i^{(2)}}
			\]
			with $(a_1,b_1,a_2,b_2)\in \NN^4$, $\left(\nu_i^{(1)}\right)\in \left(\NN_{\geq 1}\right)^{\abs{\setI_1}}$, 
			$\left(\nu_i^{(2)}\right)\in \left(\NN_{\geq 0}\right)^{\abs{\setI_2}}$, all the $\mu_i^{(1)}\in \k$ and $\mu_i^{(2)}\in \k$ are distinct.  
			We denote by $r = \abs{\setI_1}+\abs{\setI_2}$.
			We have 
			\[
				\Eu\left((P_1/P_2)^{-1}(1)\cap \Gm^2 \setminus \{P_1.P_2=0\}\right) = 
				-r \abs{\frac{2\mS_1}{\sum_{i\in \setI_1}\nu_i^{(1)}} - \frac{2\mS_2}{\sum_{i\in \setI_2}\nu_i^{(2)}}}
			\]
			with for $j \in \{1,2\}$, $\mS_j$ the area of the triangle with vertices $(0,0)$, 
			${(a_j+q\sum_{i\in \setI_j}\nu_i^{(j)},b_j)}$ and $(a_j,b_j+p\sum_{i\in \setI_j}\nu_i^{(j)})$.
		\item As in \cite[Proposition 11]{carai-antonio} we can also deduce from formula (\ref{formulethmprincipal}) an expression of the monodromy zeta function of $F_{\x,c}$.
	\end{itemize}
\end{rem}

\section{Comparison of bifurcation sets} \label{section:comparisonbifsets}
Let $(P,Q) \in \k[x,y]^2 \setminus \{\k[x]^2\cup \k[y]^2\}$, $f=[P:Q]$ and $I$ its indeterminacy locus in section \ref{setting-cas-local}.  

\begin{definition} Let $\bc$ be an indeterminate. Assume $P(0,0)=Q(0,0)=0$.
	\begin{itemize}
		\item  Let $\Sigma$ be a composition of Newton transformations. A \emph{dicritical face of $(P,Q)$ at $(0,0)$ for $\Sigma$} is a face $\gamma$ of $\N\left((P-\bc Q)_{\Sigma}\right)$, such that the face polynomial $\left((P-\bc Q)_{\Sigma}\right)_{\gamma}$ depends on $\bc$.
		\item  A \emph{dicritical face} of $(P,Q)$ at $(0,0)$ is a dicritical face for some such $\Sigma$.
	\end{itemize}
\end{definition}

\begin{notation}
 Let $\Sigma$ be a composition of Newton transformations. 
 Let $\vertex$ be a vertex of $\N\left((P-\bc Q)_{\Sigma}\right)$. 
 We consider the coefficient $c_{(P-\bc Q)_{\Sigma}}(\vertex)=c_{P_\Sigma}(\vertex)-\bc c_{Q_\Sigma}(\vertex)$.
 We say that $c_0$ cancels the coefficient of the vertex $\vertex$ or simply the vertex $\vertex$ if $c_{(P-c_0 Q)_{\Sigma}}(\vertex)=0$.
\end{notation}

\begin{notation}
	Let $\gamma$ be a one dimensional dicritical face with primitive normal vector $(p,q)$. 
		Denote by $V_{\gamma}$ the set of $c \in \AA^1$ which do not cancel the vertices of $\gamma$. 
	For any $c \in V_{\gamma}$ we denote by $R_{\gamma,c}$ the set of roots of the face polynomial $(P-cQ)_{\gamma}$. 
	For any $\mu \in R_{\gamma,c}$, we denote by $\nu_{c}(\mu)$ its multiplicity as a root of $(P-cQ)_{\gamma}$.
	We denote by $R_{V_\gamma}$ the intersection $\cap_{c \in V_\gamma} R_{\gamma,c}$ and for any $\mu \in R_{V_\gamma}$, we denote by $\nu(\mu)=\min\{ \nu_{c}(\mu) \mid c \in V_\gamma\}$. We denote by $\overline{(P-\bc Q)_{\gamma}}$ the polynomial 
	$(P-\bc Q)_{\gamma}/\prod_{ \mu \in R_{V_\gamma}}(y^p-\mu x^q)^{\nu(\mu)-1}$.
\end{notation}

\begin{definition} \label{Newtongen} 
	A value $c_0 \in \k$ is called \emph{Newton non-generic} at $(0,0)$ if there is a composition of Newton transformations $\Sigma$, such that one of the following conditions is satisfied:
	\begin{itemize}
		\item 
			there is a dicritical vertex $\vertexw$ of $(P-\bc Q)_\Sigma$, not contained in a coordinate axis  such that $c_0$ cancels $\vertexw$.
		\item 
			there is a dicritical vertex $\vertexw$ of $(P-\bc Q)_\Sigma$, contained in a coordinate axis, such that $c_0$ cancels  $\vertexw$ and either, the adjacent dimension 1 face $\gamma$ to $\vertexw$ is non-smooth, either the face $\gamma$ is smooth but the vertex $\vertex$ (in Definition \ref{def:smooth}) is cancelled by $c_0$.
		\item 
			there is a dicritical one dimensional face $\gamma$, such that $c_0$ does not cancel the vertices of $\gamma$ but cancels the discriminant of $\overline{(P-\bc Q)_{\gamma}}$: there exists $\mu \in R_{V_\gamma}$ such that $\nu_{c_0}(\mu)>\nu(\mu)$ or 
				there exists $\mu \in R_{\gamma,c_0}\setminus R_{V_\gamma}$ such that $\nu_{c_0}(\mu)>1$.
	\end{itemize}
	A value $c_0 \in \k$ is said to be \emph{Newton generic} at $(0,0)$ for $(P,Q)$ if it is not Newton non-generic.
	The set of Newton non-generic values $\B_{f,(0,0)}^{\Newton}$ is called \emph{Newton bifurcation set of $f$ at $(0,0)$}.
\end{definition}

\begin{rem} 
	For any $\x \in I$ we define by translation the Newton bifurcation set $\B_{f,\x}^{\Newton}$.
\end{rem}

\begin{prop} 
	For any $\x \in I$, the set $\B_{f,\x}^{\Newton}$ is finite. 
\end{prop}

\begin{proof} 
	This follows from the definition of Newton non-generic values because the Newton algorithm has finitely many steps and the discriminants have finitely many roots. 
\end{proof}

\begin{thm}\label{thmBfNewtonBftop}
	Assume the curves $P=0$ and $Q=0$ do not have a common irreducible component. Let $\x \in I$.
	If for any $(a,b)\in \CC^2$, $\x$ is smooth or an isolated critical point of $aP+bQ$, then 
	\[\B_{f,\x}^{\Newton} = \B_{f,\x}^{\top}.\]
\end{thm}

\begin{proof} 
	By translation, we can assume that $\x=(0,0)$. 

	We prove first that $\B_{f,(0,0)}^{\top} \subset \B_{f,(0,0)}^{\Newton}$. Indeed, if $c_1$ and $c_2$ do not belong to 
	$\B_{f,(0,0)}^{\Newton}$, the Newton algorithms of $P-c_1Q$ and $P-c_2Q$ are the same, then we obtain 
	\[\nMilnor(P-c_1Q,(0,0))=\nMilnor(P-c_2Q,(0,0))\]
	because by \cite[Theorem 5]{carai-antonio}, we have\footnote{Correcting a misprint in the formula of \cite[Theorem 5]{carai-antonio}.}
	for any $c\in \AA^1$, 
	\begin{equation} \label{formulenM} 
		\nMilnor(P-cQ,(0,0)) = 1+2\mS_{\Delta(P-cQ),P-cQ} -\eps_{b_h}a_h- \eps_{a_v}b_v 
		- \sum_{\sigma} \tilde{\chi}_c\left(\left(S_{(P-cQ)_{\sigma}}\right)_{(0,0)}^{(1)}\right)
	\end{equation}
	as a consequence of 
	\begin{equation} 
		\tilde{\chi}_{c}\left(S_{(P-cQ),(0,0)}^{(1)}\right) = -2\mS_{\Delta(P-cQ),P-cQ} + \eps_{b_h}a_h + \eps_{a_v}b_v 
		+ \sum_{\sigma} \tilde{\chi}_c\left(\left(S_{(P-cQ)_{\sigma}}\right)_{(0,0)}^{(1)}\right)
	\end{equation}
        where 
	\begin{itemize}
		\item the summation is taken over all possible Newton maps associated to $\N(P-cQ)$, 
		\item $\gamma_h=(a_h,b_h)$, $\gamma_v=(a_v,b_v)$ are the horizontal and vertical faces, 
		\item $\eps_{b_{h}}=0$ if $b_h \neq 0$ and $\eps_{b_{h}}=1$ if $b_h =0$, $\eps_{a_{v}}=0$ if $a_v \neq 0$ and $\eps_{a_{v}}=1$ if $a_v=0$,
		\item $\mS_{\Delta(P-cQ),P-cQ}$ is the area of $\Delta(P-cQ)$ relatively to $P-cQ$ defined in \cite{carai-antonio}.
	\end{itemize}
	The set $\B_{f,(0,0)}^{\Newton} \cup \B_{f,(0,0)}^{\top}$ is finite, and on its complement, by Theorem \ref{criterebfx},
	the function {${c\mapsto \nMilnor(P-c Q,(0,0))}$} is constant equal to $\nMilnor(P-c_{\gen}Q,(0,0))$. 
	If $c \notin \B_{f,(0,0)}^{\Newton}$ we deduce that 
		\[\nMilnor(P-c_{\gen}Q,(0,0)) = \nMilnor(P-cQ,(0,0))\]
	and by Theorem \ref{criterebfx}, we have $c \notin \B_{f,(0,0)}^{\top}$
	implying the result. We do the same for the value $\infty$ working with the family $Q-cP$, with $c\in \AA^1$.

	We prove now that $\B_{f,(0,0)}^{\Newton} \subset \B_{f,(0,0)}^{\top}$. Let $c_0 \notin \B_{f,(0,0)}^{\top}$. We prove that $c_0 \notin \B_{f,(0,0)}^{\Newton}$. 

        As $\B_{f,(0,0)}^{\top}$ is finite, there is a ball $B(c_0,r)$ such that $B(c_0,r)\cap \B_{f,(0,0)}^{\top} = \emptyset$. As $\B_{f,(0,0)}^{\Newton}$ is finite, there is $c \in B(c_0,r)$ such that for any $t\in ]0,1]$, $c_t:=(1-t)c_0+tc \notin \B_{f,(0,0)}^{\Newton}$. 
	By construction, for any $t \in [0,1]$, $c_t \notin \B_{f,(0,0)}^{\top}$, so the application $t \mapsto \nMilnor(P-c_tQ,(0,0))$ is constant on $[0,1]$, then by \cite{LeRamanujam, Par99a} (or \cite{Wall}[p151, Notes 6.6 and p86]), $P-c_0Q$ and $P-cQ$ are equisingular and have the same Puiseux characteristic exponents. In particular they have the same multiplicity $m$.
	
		We have $\Delta(P-c_0Q)\subset \Delta(P-\bc Q)$. As $c \notin \B_{f,(0,0)}^{\Newton}$, it does not cancel coefficients of vertices not in the axes, and up to smooth faces, the one dimensional faces of $P-cQ$ are those of $P-\bc Q$. So, up to smooth faces, $\N(P-cQ)$ is under $\N(P-c_0Q)$.

		We prove in fact that, up to smooth faces, $\N(P-cQ)$ and $\N(P-c_0Q)$ are equal. Indeed, by definition of $m$
		there is a face $\gamma_{c,m}$ of $\N(P-cQ)$ and a face $\gamma_{c_0,m}$ of $\N(P-c_0Q)$, together contained in the line $x+y=m$. 
		As $\N(P-cQ)$ is under $\N(P-c_0Q)$ and, $\Delta(P-cQ)$ and $\Delta(P-c_0Q)$ are convex, we have $\gamma_{c_0,m}\subset \gamma_{c,m}$. 
		We write the face polynomial
		\[
			(P-cQ)_{\gamma_{c,m}} = x^{A}y^{B}\prod_{i=1}^{r}(y-\mu_i x)^{\nu_i}\sum_{j=0}^{n}(a_j(P)+ca_j(Q))x^jy^{n-j}.
		\]
		If for any $j \in [\![0\:;n]\!]$, $a_j(Q)=0$, then we have $\gamma_{c,m}=\gamma_{c_0,m}$. 
		Otherwise, as $c \notin \B_{f,(0,0)}^{\Newton}$, the polynomial $\sum_{j=0}^{n}(a_j(P)+ca_j(Q))x^jy^{n-j}$ has $n$ distinct roots $\lambda_i$ different from the $\mu_i$.
		Then, by equisingularity, and parametrization of branches of the curves $P-cQ=0$ and $P-c_0Q=0$, we obtain that the polynomial
		\[\sum_{j=0}^{n}(a_j(P)+ca_j(Q))x^jy^{n-j}\] should have $n$ distinct roots, implying that $\gamma_{c,m}=\gamma_{c_0,m}$.
		Similarly, we prove the equality of the other faces because, as the Zariski characteristic pairs trees\footnote{In the Newton algorithm of $P-cQ$ in section~\ref{section:Newton-algorithm-local}, the tree of primitive normal vectors $(p_i,q_i)$ is called Zariski characteristic pairs tree of $P-cQ$ and up to vectors $(1,q)$ or $(p,1)$ of smooth faces, it is equivalent to the data of Puiseux characteristic exponents of $P-cQ$ by \cite[Notes 3.7]{Wall}.}of $P-cQ$ and $P-c_0Q$ are equal, up to smooth faces, $\N(P-cQ)$ and $\N(P-c_0Q)$ are  parallel (meaning faces are parallel two by two).
			
		We fix a face $\gamma$ of $\N(P-cQ)$ supported by a line $ap+bq=N$ with $(p,q)$ coprime integers. 
		Applying a Newton transform $\tsigma$, we have $(P-c_0 Q)_{\tsigma}=x_1^N (P-c_0Q)'$ and $(P-cQ)_{\tsigma} = x_1^N(P-cQ)'$.
		As $P-c_0Q$ and $P-cQ$ are equisingular, $(P-c_0Q)'$ and $(P-cQ)'$ are equisingular and (up to normal vectors of smooth faces) they have same Zariski characteristic pairs and same multiplicity, then as before they have the same Newton polygon up to smooth faces.
		We conclude that the face polynomials $(P-c_0Q)_{\gamma}$ and $(P-cQ)_{\gamma}$ have the same type: same common roots with same multiplicities and other (non common) roots have multiplicity one.

		By induction, this implies that $c_0 \notin \B_{f,(0,0)}^{\Newton}$.
	The base cases are immediate. 
\end{proof}

\begin{definition}[{\cite[\S 4.1.2]{Raibaut-fractions}}] \label{Bfmot}
	A value $c$ is called \emph{motivically atypical for $f$ at $\x \in I$} if $S_{f,\x,c}\neq 0$.
	The set of motivically atypical values $\B_{f,\x}^{\mot}$ is called the \emph{motivic bifurcation set of $f$ at $\x$}.
\end{definition}

\begin{theorem}[{\cite[Theorem 7]{Raibaut-fractions}}] \label{bfx}
For all $\x\in I$, the motivic bifurcation set $\B_{f,\x}^{\mot}$ is finite.
\end{theorem}

\begin{theorem} \label{thmegaliteBf}
	Assume the curves $P=0$ and $Q=0$ do not have a common irreducible component. Let $\x \in I$. If for any $(a,b)\in \CC^2$, 
	$\x$ is smooth or an isolated critical point of $aP+bQ$, then 
	\[\B_{f,\x}^{\top} = \B_{f,\x}^{\mot} = \B_{f,\x}^{\Newton}.\]
\end{theorem}
\begin{proof} 
	By Theorem \ref{thmBfNewtonBftop}, we have $\B_{f,\x}^{\Newton} = \B_{f,\x}^{\top}$.

	Let $c_0 \in \B_{f,\x}^{\top}$. By Theorem \ref{criterebfx}, $\chi_{c}(F_{\x,c_0})\neq 0$ and by Theorem \ref{realisationchic},
	$S_{f,\x,c_0}^{(1)}\neq 0$ then ${c_0 \in \B_{f,\x}^{\mot}}$. We conclude that $\B_{f,\x}^{\top}\subset \B_{f,\x}^{\mot}$.

	Up to a translation, we assume $\x=(0,0)$ and we prove now that $ \B_{f,\x}^{\mot} \subset \B_{f,\x}^{\Newton}$.
	Let $c_0\notin \B_{f,(0,0)}^{\Newton}$. Let $\bc$ be an indeterminate\footnote{We can assume $c_0\neq \infty$ working with $P-\bc Q$. In the case $c_0 = \infty$, we consider $Q-\bc P$.}. 
	As, by Definition \ref{Newtongen}, for any Newton transformation $\tsigma$, we have $c_0\notin \B_{f_{\tsigma},(0,0)}^{\Newton}$, by induction we will prove that $S_{f,(0,0),c_0}=0$, namely $c_0 \notin \B_{f,(0,0)}^{\mot}$.\\
	
	$\star$ Assume that $(P-c_0Q,Q)$ is a base case. By Remark \ref{rembasecases} we consider two cases. \smallskip

	Case 1: 
	\[(P-c_0Q)(x,y) = x^{M_1}y^{m_1}\unit_1(x,y)\:\:\text{and}\:\: Q(x,y)=x^{M_2}y^{m_2}\unit_2(x,y)\]
	with $(M_1,m_1) \in \NN^2$, $(M_2,m_2)\in \NN^{2}$ and 
	$(\unit_1,\unit_2)\in \k[x,y]^2$ with $\unit_1(0,0)\neq 0$ and $\unit_2(0,0)\neq 0$. Then, we have 
	\[(P-\bc Q)(x,y)= x^{M_1}y^{m_1}\unit_1(x,y)+(c_0-\bc) x^{M_2}y^{m_2}\unit_2(x,y).\]
	As $c_0$ is Newton generic, 
	\begin{itemize}
		\item  if $m_2\geq 1$ and $M_2 \geq 1$ then $(M_2,m_2)$ is not a dicritical vertex, then $M_2\geq M_1$ and $m_2\geq m_1$ and $S_{f,(0,0),c_0}=0$ by Example \ref{casdebase1}.
		\item Assume $m_2=0$ (similarly assuming $M_2=0$). 
	\begin{itemize}
		\item If $M_1\leq M_2$ then by Example \ref{casdebase1}, $S_{f,(0,0),c_0}=0$.
		\item If $M_2<M_1$ then $(M_2,0)$ is a dicritical vertex (and $c_0$ cancels its coefficients) but there is no one-dimensional smooth face attached to it, which contradicts $c_0 \notin \B_{f,(0,0)}^{\Newton}$.
	\end{itemize}
	\end{itemize}

	Case 2: 
	\[(P-c_0Q)(x,y) = x^{M_1}A(x,y)^{m_1}\unit_1(x,y)\:\text{and}\: Q(x,y)=x^{M_2}A(x,y)^{m_2}\unit_2(x,y)\]
	with $(M_1,m_1,M_2,m_2)\in \NN^4$, 
	$A(x,y)=y-\mu x^q+g(x,y)$, $q\in \NN_{\geq 1}$, $g(x,y)=\sum_{a+bq>q}c_{a,b}x^ay^b$ and $\unit\in \k[x,y]$ with $\unit(0,0)\neq 0$. 

	In a similar way as before, we have
	\[(P-\bc Q)(x,y)=x^{M_1}A(x,y)^{m_1}\unit_1(x,y)+(c_0-\bc) x^{M_2}A(x,y)^{m_2}\unit_2(x,y).\]
	The Newton polygons of $x^{M_1}A(x,y)^{m_1}\unit_1(x,y)$ and $(c_0-\bc) x^{M_2}A(x,y)^{m_2}\unit_2(x,y)$ have only one one-dimensional face and are parallel. As $c_0 \notin \B_{f,(0,0)}^{\Newton}$, the one dimensional face of 
	${(c_0-\bc) x^{M_2}A(x,y)^{m_2}\unit_2(x,y)}$ is not dicritical (because $c_0$ cancels all the coefficients of points of that face), implying that $M_1\leq M_2$ and $m_1 \leq m_2$ then $S_{f,(0,0),c_0}=0$ by Example \ref{casdebase2}. \\

	$\star$ Assume that $(P-c_0Q,Q)$ is not a base case. We consider formula (\ref{formulethmprincipal}).

	We prove now that 
	\[
		S_{f,(0,0),c_0}-\sum_{\tsigma} S_{f_{\tsigma},(0,0),c_0}=0
	\]
	studying the coefficients $\eps_{C_h}$, $\eps_{C_v}$ and $\eps_{C}$. Then, by an induction process, assuming 
	$S_{f_{\tsigma},(0,0),c_0}=0$ for any Newton transformation $\tsigma$, we conclude that $S_{f,(0,0),c_0}=0$.
	\begin{enumerate}
		\item We study $\eps_{C_h}$ (and as well $\eps_{C_v}$).

			Denote $(a_0,b_0)=\gamma_{h}(P-c_0Q)$, $(a_1,b_1)=\gamma_{h}(Q)$ and $(\overline{a_0},\overline{b_0})=\gamma_{h}(P-\bc Q)$.

			\begin{enumerate}
				\item
					Assume $c_0$ does not cancel the coefficient of the vertices $(\overline{a_0},\overline{b_0})$ then 
					$(a_0,b_0)=(\overline{a_0},\overline{b_0})$.

					By Proposition \ref{conedim2}, where ``\dots'' means elements not in the horizontal face, we have:

					$\bullet$ If $P=\alpha_{a,0}x^a + \dots$ and $Q=\beta_{a_1,0}x^{a_1} + \dots$ then:

					If $a<a_1$ then $P-c_0 Q=\alpha_{a,0}x^a + \dots$, then $a_0=a<a_1$ and $\eps_{C_h}=0$.

					If $a_1<a$ then $P-c_0 Q=-c\beta_{a_1,0}x^{a_1}+\dots$, then $a_0=a_1$ and $\eps_{C_h}=0$.

					If $a=a_1$ then 
					$P-c_0 Q = (\alpha_{a,0}-c_0 \beta_{a_1,0})x^{a_1}+\dots$, but 
					$\alpha_{a,0}-c_0 \beta_{a_1,0}\neq 0$ by assumption on $c_0$, then $a_0=a_1$ and $\eps_{C_h}=0$.

					$\bullet$ If $P=\alpha_{a,0}x^a+\dots$ and $Q=\beta_{a_1,b_1}x^{a_1}y^{b_1}+\dots$ with $b_1>0$ then 
					$P-c_0 Q = \alpha_{a,0}x^a+\dots$ and we have $(a_0,b_0)=(a,0)$ 
					then $\left((a_0-a_1,b_0-b_1)\mid (0,1)\right)<0$ and $\eps_{C_h}=0$.

					$\bullet$ If $P=\alpha_{a,d}x^ay^d+ \dots$ with $d>0$ and $Q=\beta_{a_1,0}x^{a_1}+\dots$ then \\
					$P-c_0 Q = -c_0 \beta_{a_1,0}x^{a_1}+\dots$ then $a_0=a_1$ and $\eps_{C_h}=0$.

					$\bullet$ If $P=\alpha_{a,d}x^ay^d + \dots$ and $Q=\beta_{a_1,b_1}x^{a_1}y^{b_1}+\dots$ then: 

					If $P-c_0 Q = \alpha_{a,d}x^ay^d+\dots$ then $(a_0,b_0)=(a,d)$, and we have either $d<b_1$, either $d=b_1$ and $a_0\leq a_1$, 
					then $\left((a_0-a_1,b_0-b_1)\mid (0,1)\right)\leq 0$, and $\eps_{C_h}=0$.

					If $P-c_0 Q = -c_0 \beta_{a_1,b_1}x^{a_1}y^{b_1}+\dots$, then $(a_0,b_0)=(a_1,b_1)$ and $\eps_{C_h}=0$.

					If $P-c_0 Q = (\alpha_{a,d}-c_0 \beta_{a_1,b_1})x^{a_1}y^{b_1}+\dots$ then by assumption on $c_0$, $\alpha_{a,d}\neq c_0 \beta_{a_1,b_1}$, then $(a_0,b_0)=(a_1,b_1)$ and $\eps_{C_h}=0$.

				\item \label{casannulation} 
					Assume $c_0$ cancels the coefficient of the dicritical face $\vertexw=(\overline{a_0},\overline{b_0})$ of $P-\bc Q$, then $\overline{b_0}=0$ and
					$(a_1,b_1)=(\overline{a_0},\overline{b_0})$. Furthermore
					the adjacent face $\gamma$ of dimension 1 (with points $\vertex$ and $\vertexw$ in Definition \ref{def:smooth}) is smooth, and $c_0$ does not cancel the coefficient of $\vertex$. 

					$\bullet$ Assume $\gamma_{h}(P-c_0 Q)=(a_0,0)$ with $a_0 > \overline{a_0}$.
					The attached compact one dimensional face denoted by $\fdelta$ has vertices $(a_0,0)$ and $\vertex=(a_\vertex,1)$ with $a_\vertex < \overline{a_0}=a_1$.
					\begin{center}
						\includegraphics[scale=0.5]{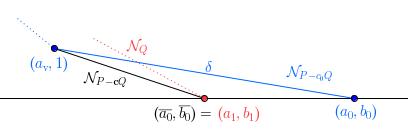}
					\end{center}

					By Theorem \ref{formuleSf}, we have
					\begin{multline*} 
						S_{f,(0,0),c_0}  =  \eps_{C_h}[x^{a_0-a_1}\colon\Gm \to \Gm, \sigma_{\Gm}] \\
						-\eps_{C_\fdelta}
						\left[(P-c_0Q)_{\gamma_{P-cQ}(C_\fdelta)}/Q_{\gamma_{Q}(C_\fdelta)}\colon\Gm^2 \setminus \{(P-c_0Q)_{\gamma_{P-c_0Q}(C_\fdelta)}\cdot Q_{\gamma_{Q}(C_\fdelta)}=0\} \to \Gm, \sigma_{C_\fdelta}\right]\\
						+ S_{f_{\sigma_{\fdelta}},(0,0),c_0} + \dots 
					\end{multline*}
					As $a_0>a_1$ we obtain $\eps_{C_h}=1$, by Proposition \ref{conedim2}. 
					We have $\gamma_{Q}(C_\fdelta)=(a_1,0)$, $(a_0,0)\in \fdelta$ and $((a_0-a_1,0)\mid \eta_{\fdelta})>0$, because $\eta_{\fdelta} \in \left(\NN_{>0}\right)^2$. 
					Then, by Proposition \ref{conedim1}, we have $\eps_{C_\fdelta}=-1$.
					As $(a_v,1)\in \fdelta$ and $(a_0,b_0)\in \fdelta$, by the Newton transformation $\tsigma$ with $\mu$ the root of $
					(P-c_{0}Q)_{\fdelta}$
					\[
						\tsigma \colon \left\{ \begin{array}{ccl}
								x & = & x_1 \\ 
								y & = & x_1^{a_0-a_v}(y_1+\mu)
							\end{array}
						\right.
					\]
					we have 
					$(P-c_0Q)_{\sigma}=x_1^{a_0}y_1 \unit_1(x,y)$ and $Q_{\sigma}=x_1^{a_1}\unit_2(x,y)$
					where $\unit_1$ and $\unit_2$ are units. Then, by the base case formula, we have 
					\[S_{f_{\sigma_{\fdelta}},(0,0),c_0}=-[x_{1}^{a_0-a_1}y_1\colon\Gm^2 \to \Gm, \sigma_{\Gm^2}].\]

					Applying cut and paste relations in the ring $\M_{\Gm}^{\Gm}$, we have 
					\begin{multline} \label{compensationintermediaire}
						\eps_{C_h}[x^{a_0-a_1}\colon\Gm \to \Gm, \sigma_{\Gm}] + S_{f_{\sigma_{\fdelta}},(0,0),c_0} \\
						-\eps_{C_\fdelta}
						\left[(P-c_0Q)_{\gamma_{P-c_0Q}(C_\fdelta)}/Q_{\gamma_{Q}(C_\fdelta)}\colon\Gm^2 \setminus \{(P-c_0Q)_{\gamma_{P-c_0Q}(C_\fdelta)}\cdot Q_{\gamma_{Q}(C_\fdelta)}=0\} \to \Gm, \sigma_{C_\fdelta}\right]
						= 0.
					\end{multline}
				     	For the convenience of the reader we explain the computation done in formula \ref{compensationintermediaire}. Indeed, we have 
					\begin{multline*}	
						\left[(P-c_0Q)_{\gamma_{P-c_0Q}(C_\fdelta)}/Q_{\gamma_{Q}(C_\fdelta)}\colon\Gm^2 \setminus \{(P-c_0Q)_{\gamma_{P-c_0Q}(C_\fdelta)}\cdot Q_{\gamma_{Q}(C_\fdelta)}=0\} \to \Gm, \sigma_{C_\fdelta}\right] \\ = 
						\left[A(x^{a_0-a_1}-Bx^{a_v-a_1}y)\colon\Gm^2 \setminus \{x^{a_0}-Bx^{a_v}y=0\} \to \Gm, \sigma_{C_\fdelta}\right].
					\end{multline*}
					for some constants $A$ and $B$. 
					The action of $\sigma_{C_\fdelta}=\sigma_{(1,a_0-a_v)}$ is 
					$\lambda.(x,y)=(\lambda x, \lambda^{a_0-a_v}y)$.
					Now, considering the equivariant toric changes of variables
					\[ (\Gm^2,\sigma_{(1,a_0-a_v)}) \to (\Gm^2,\sigma_{(1,a_0-a_1)}), (x,y) \mapsto (x,z=Bx^{a_v-a_1}y)\]
					and 
					\[(\Gm^2\setminus \{x^{a_0-a_1}=z\},\sigma_{(1,a_0-a_1)}) \to (\Gm^2 \setminus \{u=x^{a_0-a_1}\},\sigma_{(1,a_0-a_1)}), (x,z) \mapsto (x,u=A(x^{a_0-a_1}-z))\]
					we have the equalities by isomorphisms
					\begin{multline*}
						\left[A(x^{a_0-a_1}-Bx^{a_v-a_1}y)\colon\Gm^2 \setminus \{x^{a_0}-Bx^{a_v}y=0\} \to \Gm, \sigma_{C_\fdelta}\right]
						\\ =
						\left[A(x^{a_0-a_1}-z)\colon\Gm^2 \setminus \{x^{a_0-a_1}= z \} \to \Gm, \sigma_{(1,a_0-a_1)}\right]
						\\ =
						\left[u\colon \Gm^2 \setminus \{u=x^{a_0-a_1}\} \to \Gm, \sigma_{(1,a_0-a_1)}\right].
					\end{multline*}
					By cut and paste relations we have 
					\begin{multline*}
						\left[u\colon \Gm^2 \setminus \{u=x^{a_0-a_1}\} \to \Gm, \sigma_{(1,a_0-a_1)}\right] \\
						= \left[u\colon \Gm^2 \to \Gm, \sigma_{(1,a_0-a_1)}\right] 
						- \left[u\colon \Gm^2 \cap \{u=x^{a_0-a_1}\} \to \Gm, \sigma_{(1,a_0-a_1)}\right].
					\end{multline*}
					We have the equivariant isomorphism
					\[ (\Gm,\sigma_{\Gm}) \to (\Gm^2 \cap \{u=x^{a_0-a_1}\} \to \Gm, \sigma_{(1,a_0-a_1)}), x \mapsto (x,x^{a_0-a_1}) \]
					then we have
					\[ \left[u\colon \Gm^2 \cap \{u=x^{a_0-a_1}\} \to \Gm, \sigma_{(1,a_0-a_1)}\right] = [x^{a_0-a_1}\colon\Gm \to \Gm, \sigma_{\Gm}].\]
					Also, we have 
					\[\left[u\colon \Gm^2 \to \Gm, \sigma_{(1,a_0-a_1)}\right] = (\LL-1)[u\colon\Gm\to \Gm, \sigma_(a_0-a_1)]=(\LL-1)\]
					because in the construction of the ring $\M_{\Gm}^{\Gm}$, we have 
					\[
						[u\colon\Gm\to \Gm, \sigma_{a_0-a_1}] = [u\colon\Gm\to \Gm, \sigma_{\Gm}].
					\]

					As well, we have the equivariant isomorphism (above $\Gm$)

					\[ (\Gm^2,\sigma_{\Gm^2}) \to (\Gm^2, \sigma_{(1,a_0-a_1+1)}), (x_1,y_1) \mapsto (x_1,z=x_1^{a_0-a_1}y_1) \]
					then we have
					\[
						[x_1^{a_0-a_1}y_1 \colon \Gm^2 \to \Gm, \sigma_{\Gm^2}] = [u\colon \Gm^2 \to \Gm, \sigma_{(1,a_0-a_1+1)}] = \LL-1.
					\]
					The formula \ref{compensationintermediaire} follows now by a direct computation. 

					$\bullet$ Assume $\gamma_{h}(P-c_0 Q)=\vertex$. 
					As $(a_1,b_1)=(\overline{a_0},\overline{b_0})$, we have $\overline{a_0}=a_1 > a_0$ and $\eps_{C_h}=0$.
			\end{enumerate}
			\smallskip

		\item  We prove now that  $\eps_{C}=0$ for each dimension 2 cone $C \in \E_{c_0}$ different from $C_h$ and $C_v$.\\
			We denote by $\gamma_{P-c_0Q}(C)=(a_0,b_0)$ and $\gamma_{Q}(C)=(a_1,b_1)$ the associated 0-dimensional faces in $\N_Q$ and $\N_{P-c_0Q}$. We denote by 
			\[C_{\gamma_{Q}(C)} = \RR_{>0}\eta_1+\RR_{>0}\eta_2\:\:\text{and}\:\:
				C_{\gamma_{P-c_0Q}(C)}=\RR_{>0}\omega_1 + \RR_{>0}\omega_2
			\]
			the associated cones in the fan of $\N_Q$ and $\N_{P-c_0Q}$. 
			We have $C = C_{\gamma_{Q}(C)} \cap C_{\gamma_{P-c_0Q}(C)}$.

			As $c_0$ is Newton generic, $\Delta'(Q) \subset \Delta'(P-c_0Q)$ where $\Delta'$ is defined in Notation \ref{defDelta'}.
			\begin{enumerate}
				\item If $(a_0,b_0)=(a_1,b_1)$ then by Proposition \ref{conedim2}, we have $\eps_{C}=0$.
				\item If $(a_0,b_0)\neq (a_1,b_1)$ and $(a_0,b_0) \in \N_Q$ then as $\Delta'(Q) \subset \Delta'(P-c_0Q)$, by convexity, we get that the intersection $C_{\gamma_{Q}(C)} \cap C_{\gamma_{P-c_0Q}}=C$ is empty. Contradiction.
				\item \label{casspecificque} 
					Assume that $(a_0,b_0)\notin \N_{Q}$ and $\omega_1 \in C_{\gamma_{Q}(C)}$ (the case $\omega_2 \in C_{\gamma_{Q}(C)}$ is similar). 

					Let $D$ be the line perpendicular to $\RR \omega_1$ with $(a_0,b_0)\in D$. By Proposition \ref{conedim2}:
					\begin{itemize}
						\item If $(a_1,b_1) \in D$, then $((a_0-a_1,b_0-b_1)\mid \omega_1)=0$ and $\eps_{C}=0$.
						\item If $D \cap \N_Q = \emptyset$ then $((a_0-a_1,b_0-b_1)\mid \omega_1)<0$ and $\eps_{C}=0$.
						\item As $\Delta'(Q) \subset \Delta'(P-c_0 Q)$, the case $D \cap \N_Q \neq \emptyset$ does not occur. 
					\end{itemize}
					\begin{center}
						\includegraphics[scale=0.4]{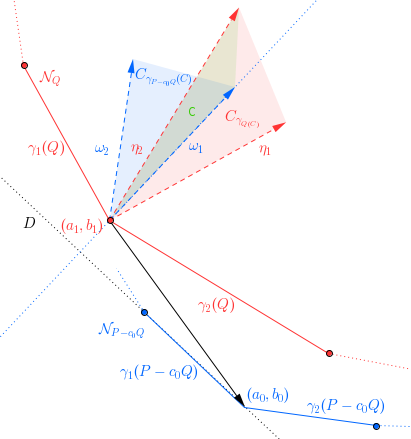}
						\includegraphics[scale=0.4]{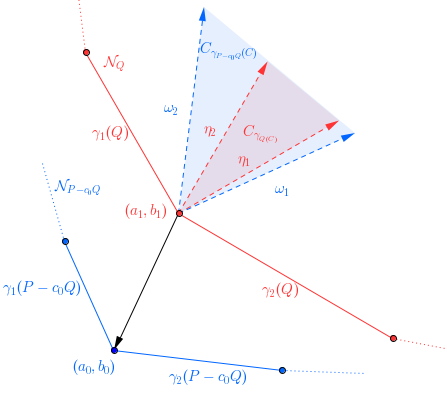}
					\end{center}

				\item 
					Assume that  $(a_0,b_0)\notin \N_{Q}$ and $(\eta_1,\eta_2) \in C_{\gamma_{P-c_0Q}(C)}^2$. Let $\gamma_1(Q)$ and $\gamma_2(Q)$ be the 1-dimensional faces of $\N_{Q}$ which intersect in $(a_1,b_1)$, with normal vectors $\eta_1$ and $\eta_2$.

					If $b_1 \neq 0$ (similarly $a_1 \neq 0$) then $(a_1,b_1)\in \Gamma'(Q) \subset \Gamma'(P-c_0 Q)$, then $(a_0,b_0)$ does not belong to the cone generated by $\gamma_1$ and $\gamma_2$ with vertex $(a_1,b_1)$,
					we get \[((a_0-a_1,b_0-b_1)\mid \eta_1)<0\:\:\text{and}\:\:((a_0-a_1,b_0-b_1)\mid \eta_2)<0\]
					implying that $\eps_{C}=0$ by Proposition \ref{conedim2}.

					If $b_1 = 0$ (similarly $a_1=0$) then $\eta_1 = (0,1)$ or $\eta_2 =(0,1)$, so $(0,1)\in C$ namely $C=C_h$ implying a contradiction with the assumption on $C$.
			\end{enumerate}

			\item We prove that $\eps_{C}=0$ for each dimension 1 cone $C \in \E_{c_0}$. By above notation, we have 
				\[C = C_{\gamma_{\gamma_{Q}(C)}} \cap C_{\gamma_{P-c_0 Q}(C)}.\]

			$\bullet$ If $C_{\gamma_{\gamma_{Q}(C)}}$ has dimension 2 and $C_{\gamma_{P-c_0 Q}(C)}$ has dimension 1, generated by a vector $\omega$, then 
			$C_{\gamma_{P-c_0 Q}(C)} \subset C_{\gamma_{\gamma_{Q}(C)}}$. 
			The face $\gamma_{Q}(C)$ has dimension 0, denoted by $(a_1,b_1)$, and the face $\gamma_{P-c_0 Q}(C)$ has dimension 1 and we fix a vertex $(a_0,b_0)$ on it. We can assume that we are not in the situation of $ \gamma_{P-c_0 Q}(C) = \fdelta$ (as in \ref{casannulation}), which is ever treated. Then similarly to the case \ref{casspecificque}, we obtain that $((a_0-a_1,b_0-b_1)\mid \omega)<0$ and by Proposition \ref{conedim1}, we get $\eps_{C}=0$.

			$\bullet$ If $C_{\gamma_{P-c_0 Q}(C)}$ has dimension 2 and $C_{\gamma_{\gamma_{Q}(C)}}$ has dimension 1 generated by a vector $\omega$, then $C_{\gamma_{\gamma_{Q}(C)}} \subset C_{\gamma_{P-c_0 Q}(C)}$. 
			The face $\gamma_{P-c_0 Q}(C)$ has dimension 0 denoted by $(a_0,b_0)$ and the face $\gamma_{Q}(C)$ has dimension 2 and we fix a vertex $(a_1,b_1)$ on it. As $\Gamma'(Q)\subset \Gamma'(P-c_0 Q)$ we have 
			$((a_0-a_1,b_0-b_1)\mid \omega)\leq 0$ then by Proposition \ref{conedim1}, we get $\eps_{C}=0$.

			$\bullet$ If $C_{\gamma_{P-c_0 Q}(C)}$ and $C_{\gamma_{\gamma_{Q}(C)}}$ have dimension 1, then they are equal (because their intersection has also dimension 1) and in the same way as before, we get $\eps_{C}=0$.
	\end{enumerate}
\end{proof}
\begin{example} \label{exemple}
	We consider $P=(x^2+y^3)^2+x^5$ and $Q=2x^4+x^2y^3+xy^5+y^8$.
	We have 
	\[ P-\bc Q = (1-2\bc)x^4+(2-\bc)x^2y^3+y^6 + x^5 - \bc xy^5-\bc y^8. \]

	We work at $(0,0)$.
	The Newton polygon $\N( P-\bc Q)$ has three faces with face polynomials
	\[y^6, (1-2\bc)x^4+(2-\bc)x^2y^3+y^6\:\:\text{and}\:\:(1-2\bc)x^4.\]
	We denote by $\gamma$ the one dimension face of $\N( P- \bc Q)$. The faces $\gamma$ and $(4,0)$ are dicritical.\\
	The value $c=1/2$ is a non Newton generic value.\\
	In a similar way, working with $Q-\bc P$, $c= \infty$ is a non Newton generic value.\\
	The face polynomial 
	\[P_{\gamma,c}(x,y) = (1-2c)x^4+(2-c)x^2y^3+y^6\]
	does not have simple roots if and only if $c \in \{-4,0\}$.\\
	If $c \notin \{-4,0\}$ the face $\gamma$ is non degenerate and there is no more non Newton generic values.\\
	By Theorem \ref{thmegaliteBf}, we conclude that $\B_{f,(0,0)}^{\top}=\{-4,0,1/2,\infty\}$.\\ 
	Using \href{https://www.singular.uni-kl.de/}{Singular}, we obtain
	\[ \nMilnor(P-c_{\gen} Q,(0,0))=15, \nMilnor(P+4Q,(0,0))=16, \nMilnor(P-Q/2,(0,0))=17,\]
	\[\nMilnor(P,(0,0))=18\:\: \text{and}\:\: \nMilnor(Q,(0,0))=16.\]
	We check that, computing the motivic invariant $S_{f,(0,0),c}$ using Theorem \ref{thmSf00c} and the pictures:
	\begin{center}
		\begin{tabular}{ccc}
			\includegraphics[scale=0.34]{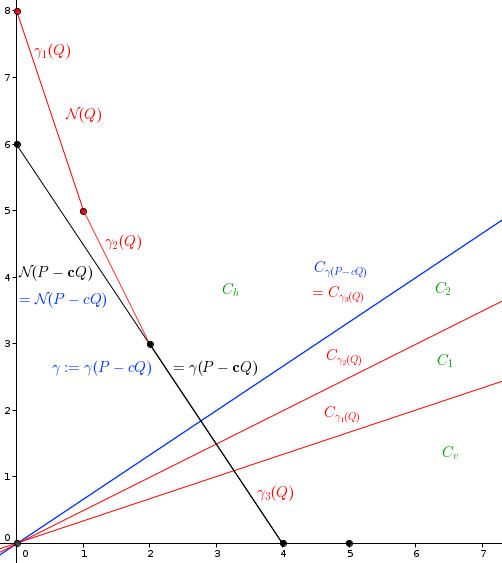} & \hspace{1cm} &  \includegraphics[scale=0.34]{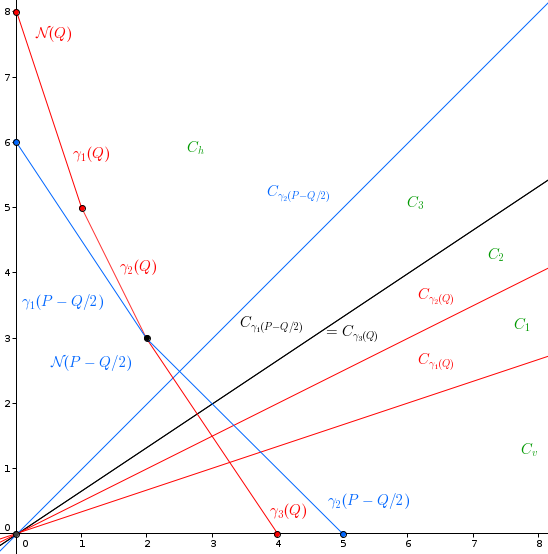} \\
			$c \neq \frac{1}{2}$ & & $c=\frac{1}{2}$
		\end{tabular}
	\end{center}

	$\bullet$ For $c \notin \{-4,0,1/2,\infty\}$, we have $\N(P-\bc Q) = \N(P-cQ)$. \\
	With notation on the picture ``$c\neq \frac{1}{2}$'', the dual fan $\E_c$ in Definition \ref{defEc} induced by $P-cQ$ and $Q$, has 
	the following cones
	\[
		C_{\gamma_{1}(Q)} = \RR_{>0}(3,1),\: C_{\gamma_{2}(Q)}=\RR_{>0}(2,1)\:\:\text{and}\:\: 
		C_{\gamma_{3}(Q)}=C_{\gamma(P-cQ)}=\RR_{>0}(3,2)
	\]
	as dimension 1 cones and 
	\[ C_v=\RR_{>0}(1,0)+\RR_{>0}(3,1), C_{1}=\RR_{>0}(3,1)+\RR_{>0}(2,1),\]
	\[C_{2}=\RR_{>0}(2,1)+\RR_{>0}(3,2)\:\:\text{and}\:\: C_h = \RR_{>0}(3,2)+\RR_{>0}(0,1)\]
	as dimension 2 cones.\\
	Applying Proposition \ref{conedim2} and Proposition \ref{conedim1}, we obtain in formula \ref{formulethmprincipal},
	\[\eps_{C_h}=\eps_{C_v}=\eps_{C_1}=\eps_{C_2}=\eps_{C_{\gamma_{1}(Q)}}=\eps_{C_{\gamma_{2}(Q)}}=\eps_{C_{\gamma_{3}(Q)}}=0.\]
	Writing the face polynomials 
	\[(P-cQ)_{\gamma}=(y^3-\mu_1x^2)(y^3-\mu_2x^2)\]
	(with $\mu_1$ and $\mu_2$ different from $-2$) 
	\[Q_{\gamma_1}=y^5(y^3+x), Q_{\gamma_2}=xy^3(y^2+x)\:\:\text{and}\:\:Q_{\gamma_3}=x^2(2x^2+y^3),\]
	we have by formula \ref{formulethmprincipal}:
	\[
		S_{f,(0,0),c} = S_{f_{\tsigma(3,2,\mu_1)},(0,0),c} + S_{f_{\tsigma(3,2,\mu_2)},(0,0),c} + 
		S_{f_{\tsigma(3,1,-1)},(0,0),c} + S_{f_{\tsigma(2,1,-1)},(0,0),c} + S_{f_{\tsigma(3,2,-2)},(0,0),c}.
	\]
	Each Newton transform is a base case. By Example \ref{casdebase2}, we have  
	\[
		S_{f_{\tsigma(3,1,-1)},(0,0),c} = S_{f_{\tsigma(2,1,-1)},(0,0),c} = S_{f_{\tsigma(3,2,-2)},(0,0),c} = 0.
	\]
	By Remark \ref{caracteristiqueEulercasdebase2}\ref{casspecifique}, we have 
	\[
		S_{f_{\tsigma(3,2,\mu_1)},(0,0),c} = S_{f_{\tsigma(3,2,\mu_2)},(0,0),c}=0 
	\]
	and we conclude that ${S_{f,(0,0),c} = 0}$.\medskip

	$\bullet$ For $c=-4$, we have $\N(P+4Q) = \N(P-\bc Q)$ and $(P+4Q)_{\gamma} = (y^3+3x^2)^2$. We obtain
	\[
		S_{f,(0,0),-4} = S_{f_{\tsigma(3,2,-3)},(0,0),-4}
	\]
	as before, with 
	\[(P+4Q)_{\tsigma(3,2,-3)}(x_1,y_1)=cx_1^{12}(y_1^2-lx_1+g(x_1,y_1))\]
	with $l \in \ZZ$ and ${g(x_1,y_1)=\sum_{2a+b>2}c_{a,b}x^ay^b}$, and 
	\[Q_{\tsigma(3,2,-3)}(x_1,y_1)=x_1^{12}\unit(x_1,y_1)\] with $\unit$ a unit.
	We have
	\begin{multline*}
		S_{f_{\tsigma(3,2,-3)},(0,0),-4} = [x_1 \colon \Gm \to \Gm, \sigma_{\Gm}] + [x_{1}^{12}y_1^2/x_{1}^{12} \colon \Gm^2 \to \Gm]\\
		+ [x_1^{12}(y_1^2-lx_1)/x_1^{12} \colon \Gm^2 \setminus (y_1^2=lx_1) \to \Gm] + S_{(f_{\tsigma(3,2,-3)})_{\tsigma(2,1,l)},(0,0),-4} 
	\end{multline*}
	by Theorem \ref{thmSf00c}, with $S_{(f_{\tsigma(3,2,-3)})_{\tsigma(2,1,l)},(0,0),-4}$ computed in Example \ref{casdebase2} with 
	\[
		\left\{
			\begin{array}{lcl}
				\left((P+4Q)_{\tsigma(3,2,-3)}\right)_{\tsigma(2,1,l)}(x_2,y_2) & = & c'x_2^{26}(y_2+l'x_2+g'(x_2,y_2))\\
				\left(Q_{\tsigma(3,2,-3)}\right)_{\tsigma(2,1,l)}(x_2,y_2)& = & x_2^{24}\unit'(x_2,y_2)
			\end{array}
		\right.
	\]
	with $(c',l')\in \ZZ^2$, $g'(x_2,y_2)=\sum_{a+b>1}c_{a,b}x_2^ay_2^b$ and $\unit'$ a unit.\\

	By \cite[Proposition 2.4]{carai-infini}, additivity of the Euler characteristic $\tilde{\chi_{c}}$, and Remark \ref{caracteristiqueEulercasdebase2}\ref{eulercasspecifique}, we have
	\[\tilde{\chi_{c}}\left( S_{f,(0,0),-4}^{(1)} \right) = 1 + 0 - 2 + 0 = -1.\]
	In particular, by Theorem \ref{Mfc}, we check
	\[
		\nMilnor(P-c_{\gen} Q,(0,0))-\nMilnor(P+4Q,(0,0))=-1.
	\] 

	$\bullet$ For $c=0$, we have $\N(P) = \N(P-\bc Q)$ and $P_{\gamma} = (y^3+x^2)^2$. As before, we obtain
	\[
		S_{f,(0,0),0} = S_{f_{\tsigma(3,2,-1)},(0,0),0}
	\]
	with 
	\[
		P_{\tsigma(3,2,-1)}(x_1,y_1)=cx_1^{12}(y_1^2-lx_1^3+g(x_1,y_1))\:\:\text{and}\:\: 
		Q_{\tsigma(3,2,-1)}(x_1,y_1)=x_1^{12}\unit(x_1,y_1)
	\]
	with $l \in \ZZ$, $g(x_1,y_1)=\sum_{2a+3b>6}c_{a,b}x^ay^b$, and $\unit$ a unit.\\

	By Theorem \ref{thmSf00c}, we obtain
	\begin{multline*}
		S_{f_{\tsigma(3,2,-1)},(0,0),0} = [x_1^3 \colon \Gm \to \Gm, \sigma_{\Gm}] + [x_{1}^{12}y_1^2/x_{1}^{12} \colon \Gm^2 \to \Gm]\\
		+ [x_1^{12}(y_1^2-lx_1^3)/x_1^{12} \colon \Gm^2 \setminus (y_1^2=lx_1^3) \to \Gm] + 
		S_{(f_{\tsigma(3,2,-1)})_{\tsigma(2,3,l)},(0,0),0} 
	\end{multline*}
	with $S_{(f_{\tsigma(3,2,-1)})_{\tsigma(2,3,l)},(0,0),0}$ computed in Example \ref{casdebase2} with 
	\[
		\left(P_{\tsigma(3,2,-1)}\right)_{\tsigma(2,3,l)}(x_2,y_2)  =  c'x_2^{30}(y_2-l'x_2^3+g'(x_2,y_2)),
		\left(Q_{\tsigma(3,2,-1)}\right)_{\tsigma(2,3,l)}(x_2,y_2) = x_2^{24}\unit(x_2,y_2)
	\]
	with $g'(x,y)=\sum_{p+3q>3}c_{a,b}x^ay^b$, $(c',l')\in \ZZ^2$ and $\unit$ a unit. As above, we check
		\[ \nMilnor(P-c_{\gen} Q,(0,0))-\nMilnor(P,(0,0)) =\tilde{\chi_{c}}\left( S_{f,(0,0),0}^{(1)} \right) = 3 + 0 - 6 + 0 = -3.\]

	$\bullet$ For $c=1/2$. With notation on the picture ``$c= \frac{1}{2}$'', the dual fan $\E_c$ in Definition \ref{defEc} induced by 
	$P-Q/2$ and $Q$, has the following cones
	\[C_{\gamma_{1}(Q)} = \RR_{>0}(3,1), C_{\gamma_{2}(Q)}=\RR_{>0}(2,1),\] 
	\[C_{\gamma_{3}(Q)}=C_{\gamma_{1}(P-Q/2)}=\RR_{>0}(3,2)\:\:\text{and}\:\:C_{\gamma_{2}(P-Q/2)}=\RR_{>0}(1,1)\]
	as dimension 1 cones and 
	\[
		C_v=\RR_{>0}(1,0)+\RR_{>0}(3,1), C_{1}=\RR_{>0}(3,1)+\RR_{>0}(2,1)
	\]
	\[
		C_{2}=\RR_{>0}(2,1)+\RR_{>0}(3,2), C_{3}=\RR_{>0}(3,2)+\RR_{>0}(1,1), C_h = \RR_{>0}(1,1)+\RR_{>0}(0,1).
	\]
	as dimension 2 cones.\\

	Applying Proposition \ref{conedim2} and Proposition \ref{conedim1}, we obtain in formula \ref{formulethmprincipal}:
	\[
		\eps_{C_v}=\eps_{C_1}=\eps_{C_2}=\eps_{C_3}=\eps_{C_{\gamma_{1}(Q)}}=\eps_{C_{\gamma_{2}(Q)}}=\eps_{C_{\gamma_{3}(Q)}}=0, \eps_{C_h}=1, \eps_{C_{\gamma_2(P-Q/2)}}=-1.
	\]
	Writing the face polynomials $(P-Q/2)_{\gamma_1}=y^3(y^3+3/2x^2)$, 
	\[(P-Q/2)_{\gamma_2}=3/2x^2(y^3+2/3x^3)=3/2x^2(y-\xi_1 x)(y-\xi_2 x)(y-\xi_3x),\] 
	${Q_{\gamma_1}=y^5(y^3+x)}$, $Q_{\gamma_2}=xy^3(y^2+x)$ and 
	$Q_{\gamma_3}=x^2(2x^2+y^3)$, we have by formula \ref{formulethmprincipal}:
	\begin{align*}
		S_{f,(0,0),1/2} & = [x\colon \Gm \to \Gm, \sigma_{\Gm}] + 
		[x^{-2}(y^3+2/3x^3) \colon \Gm^{2}\setminus (y^3+2/3x^3=0) \to \Gm, \sigma_{C} ]\\
		& + S_{f_{\tsigma(3,2,-3/2)},(0,0),1/2} +
		\sum_{i=1}^{3} S_{f_{\tsigma(1,1,\xi_i)},(0,0),1/2}   \\
		& +  S_{f_{\tsigma(3,1,-1)},(0,0),1/2} + S_{f_{\tsigma(2,1,-1)},(0,0),1/2} + S_{f_{\tsigma(3,2,-2)},(0,0),1/2}
	\end{align*}
	with $C=C_{\gamma_2(P-Q/2)}$.
	By Example \ref{casdebase2}, we have 
	\[
		S_{f_{\tsigma(3,2,-3/2)},(0,0),1/2} = S_{f_{\tsigma(3,1,-1)},(0,0),1/2} = S_{f_{\tsigma(2,1,-1)},(0,0),1/2} 
		= S_{f_{\tsigma(3,2,-2)},(0,0),1/2} = 0.
	\]
	For each $i \in \{1,2,3\}$, by Remark \ref{caracteristiqueEulercasdebase2}\ref{eulercasspecifique}, 
	we also have $\chi(S_{f_{\tsigma(1,1,\xi_i)},(0,0),1/2})=0$.\\

	So, by \cite[Proposition 2.4]{carai-infini}, we check 
	\[ \nMilnor(P-c_{\gen} Q,(0,0))-\nMilnor(P-Q/2,(0,0)) =\tilde{\chi_{c}}\left( S_{f,(0,0),1/2}^{(1)} \right) = 1 - 3 + 0 = -2.\]

	$\bullet$ For $c=\infty$, we work with $Q-\bc P$ and then with the fraction $Q/P$ and the couple $(Q,P)$. We use again the first diagram and formula (\ref{formulethmprincipal}), we have 
	\[\eps_{C_v}=1, \eps_{C_{\gamma_{1}(Q)}}=\eps_{C_{\gamma_{2}(Q)}}=-1\]
	and 
	\[\eps_{C_h}=\eps_{C_1}=\eps_{C_2}=\eps_{C_{\gamma_{3}(Q)}}=0.\]
	As before, with $C=\RR_{>0}(3,1)$, we obtain 
	\begin{align*}
		S_{f,(0,0),\infty}  = & [y^2\colon \Gm \to \Gm, \sigma_{\Gm}] + 
		[(xy^5+y^8)/y^6\colon \Gm^2\setminus (x+y^3=0) \to \Gm, \sigma_{C}] \\ \\
		+ & [(xy^5+x^2y^3)/y^6\colon \Gm^2\setminus (x+y^2=0) \to \Gm, \sigma_{C}] \\ \\
		+ & S_{f_{\tsigma(3,2,-1)},(0,0),\infty} + 
		S_{f_{\tsigma(3,1,-1)},(0,0),\infty} + S_{f_{\tsigma(2,1,-1)},(0,0),\infty} + S_{f_{\tsigma(3,2,-2)},(0,0),\infty}.
	\end{align*}
        By Remark\ref{caracteristiqueEulercasdebase2}\ref{casspecifique}, we have 
	\[S_{f_{\tsigma(3,2,-1)},(0,0),\infty}=0.\]
	The motives $S_{f_{\tsigma(3,1,-1)},(0,0),\infty}$ and $S_{f_{\tsigma(2,1,-1)},(0,0),\infty}$ are computed using 
	Example \ref{casdebase2}.\\ \\
	By Remark\ref{caracteristiqueEulercasdebase2}\ref{eulercasspecifique}, we have 
	\[\tilde{\chi_{c}}(S_{f_{\tsigma(3,1,-1)},(0,0),\infty})=\tilde{\chi_{c}}(S_{f_{\tsigma(2,1,-1)},(0,0),\infty})=0.\]
	Using computation in the previous case $c=0$, and Theorem \ref{thmSf00c}, we obtain
	\[S_{f_{\tsigma(3,2,-1)},(0,0),\infty} = S_{(f_{\tsigma(3,2,-1)})_{\tsigma(2,3,l)},(0,0),\infty}=0.\]
	Using \cite[Proposition 2.4]{carai-infini}, we check
	\[ \nMilnor(P-c_{\gen} Q,(0,0))-\nMilnor(Q,(0,0)) =\tilde{\chi_{c}}\left( S_{f,(0,0),\infty}^{(1)} \right) = 2-2-1=-1.\]
\end{example}

\section{Proof of Theorem \ref{thmSf00c}} \label{proofthmSf00c} \label{sectionpreuvethm}
In this section, using notation of sections \ref{notations:Xeps-omega} and \ref{section-Sfxc}, we detail the proof of Theorem \ref{thmSf00c}.
\subsection{Decomposition of $\left(Z_{\f,\omega,X}^{\delta}(T)\right)_{((0,0),c)}$ along $\E_c$} \label{sec:decompositionzeta}~

In this section, we introduce for each cone $C \in \E_c$ a motivic zeta function $Z_{f-c,\omega,\C}^{\delta}(T)$ and we use it to decompose the motivic zeta function $\left(Z_{\f,\omega,X}^{\delta}(T)\right)_{((0,0),c)}$ in formula~(\ref{zetadecomposition}).
\begin{rem} Let $n\geq 1$, $k\geq 1$ and $\delta\geq 1$.
	We will simply write $X_{n,k,c}^{\delta}$ for $X_{n,k,(0,0),c}^{\delta}$ defined in formula~(\ref{Xepsnk}).
	The origin of each arc of $X_{n,k,c}^{\delta}$ is $((0,0),c)$ and its generic point belongs to $X$, so from the definition of $\X$ and Remark \ref{rem:arc-carte}, there is an isomorphism between $X_{n,k,c}^{\delta}$ and 

			\[
				\left\{
					\begin{array}{c|l} 
						(x(t),y(t))\in \L(\AA^2_\k) &
						\begin{array}{l} 
							\Ord{x(t)}= k, \Ord{y(t)}>0, 
							\Ord{\omega(x(t),y(t))}=(\nu-1)k \\
							\ord{\frac{P}{Q}(x(t),y(t))-c}=n, \Ord{Q(x(t),y(t))}\leq n\delta
						\end{array} 
					\end{array}
				\right\}
			\]
			
			\noindent endowed with the map 
			$\ac{f-c} \colon (x(t),y(t)) \mapsto \ac{(f-c)(x(t),y(t))}$
			and the $\Gm$-standard action on arcs. In the following, we will identify these arc spaces.
\end{rem}
\begin{notation}
	For any cone $C$ of the fan $\E_c$ and any $(n,k)\in \left(\NN_{\geq 1}\right)^2$, we consider 
	\[X_{n,k,c}^{\delta}(\C) = \left\{(x(t),y(t)) \in  X_{n,k,c}^{\delta} \mid (\text{ord}\: x(t), \text{ord}\: y(t))\in \C \right\}\]
	endowed with its structural map $\ac{f-c}$ to $\Gm$ and we decompose  
	\[
		X_{n,k,c}^{\delta}= \bigsqcup_{\C \in \E_c} X_{n,k,c}^{\delta}(\C).
	\]
	We consider the motivic zeta function relative to the cone $\C$
	\[
		Z_{f-c,\omega,\C}^{\delta}(T) = 
		\sum_{n\geq 1} \left(\sum_{k \geq 1} \LL^{-(\nu-1)k} \mes(X_{n,k,c}^{\delta}(\C))\right)
		T^n 
	\]
	and we have the decomposition
	\begin{equation} \label{zetadecomposition} 
		\left(Z_{\f,\omega,X}^{\delta}(T)\right)_{((0,0),c)} = \sum_{\C \in \E_c} Z_{f-c,\omega,\C}^{\delta}(T).
	\end{equation}
\end{notation}
\begin{rem}\label{rem:n-inferieur-m(ord{x},ord{y})} \label{remannulation}
   For any arc $(x(t),y(t)) \in X_{n,k,c}^{\delta}(\C)$, we can write
   \[(P-cQ)(x(t),y(t))=t^{\m_{P-cQ}(\text{ord} x(t),\text{ord} y(t) )}\widetilde{(P-cQ)}(x(t),y(t),t)\]
   and
   \[Q(x(t),y(t))=t^{\m_{Q}(\text{ord} x(t),\text{ord} y(t))}\widetilde{Q}(x(t),y(t),t)\]
   with $(\widetilde{P-cQ})(x,y,u)$ and $\widetilde{Q}(x,y,u)$ in $\k[x,y,u]$ and $\m_{P-cQ}$ and $\m_{Q}$ defined in Proposition \ref{prop:dualfan-Newton-local} relatively to $\Delta(P-cQ)$ and $\Delta(Q)$. Furthermore, we have the equivalences
	   \begin{equation}\label{acconditions1}
		   \ord{(P-cQ)(x(t),y(t))} = \m_{P-cQ}(\ord{x(t)},\ord{y(t)} ) \iff 
			   (P-cQ)_{\gamma_{P-cQ}(\C)}(\ac{x(t)}, \ac{y(t)}) \neq 0
	   \end{equation}
	   \begin{equation}\label{acconditions2}
		   \ord{Q(x(t),y(t))} = \m_{Q}(\ord{x(t)},\ord{y(t)} ) \iff 
			   Q_{\gamma_{Q}(\C)}(\ac{x(t)}, \ac{y(t)}) \neq 0.
	   \end{equation}
\end{rem}

	In the following, using the angular components conditions (\ref{acconditions1}) and (\ref{acconditions2}), we decompose in formula (\ref{decompzetaomega}) each motivic zeta function $Z_{f-c,\omega,\C}^{\delta}(T)$.

\begin{rem}\label{rem:finitudeml} Let $(x(t),y(t)) \in X_{n,k,c}^{\delta}(\C)$ be an arc. We denote by 
	\[m = \Ord\!{(P-cQ)(x(t),y(t))}\:\:\text{and}\:\: l = \Ord\!{Q(x(t),y(t))}.\]
	Then, we have $n=m-l$ and $l\leq n\delta$. In particular, there are finitely many such pairs $(m,l)$ when $(x(t),y(t))$ runs over
	$X_{n,k,c}^{\delta}(\C)$.
\end{rem}
\begin{notation} \label{NotationsXmlkc}For any $(m,l,k)\in (\Z_{\geq 1})^3$, with $m>l$, we define 
	\[ X_{(m,l),k,c}^{\delta}(\C)=
		\left\{
				(x(t),y(t))\in X_{m-l,k,c}^{\delta}(\C) \mid \Ord\!{(P-cQ)(x(t),y(t))}=m, \Ord\!{Q(x(t),y(t))}=l
		\right\}.
	\]
\end{notation}

	Using Remark \ref{rem:finitudeml} and the additivity of the motivic measure, we have the decomposition
	\[Z_{f-c,\omega,\C}^{\delta}(T) = \sum_{n\geq 1} 
		\big(\:\sum_{\underset{m-l=n, l\leq n \delta}{(m,l,k)\in (\Z_{\geq 1})^3}}\LL^{-(\nu-1)k}
		\mes \big(X_{(m,l),k,c}^{\delta}(\C)\big) \:\big)T^n.
	\]
\begin{notation}
	For any $(\blacktriangle,\blacktriangledown) \in \{=,\neq\}^2$, we define
	\begin{equation} \label{decompositiontriangles}
		Z_{f-c,\omega,\C}^{\delta,(\blacktriangle,\blacktriangledown)}(T) = 
		\sum_{n\geq 1} \big(\: \sum_{\underset{m-l=n, l\leq n\delta}{(m,l,k)\in (\Z_{\geq 1})^3}}\LL^{-(\nu-1)k}
		\mes\big(X_{(m,l),k,c}^{\delta,(\blacktriangle,\blacktriangledown)}(\C)\big)\:\big)T^n 
	\end{equation}
	where for any $(m,l,k)$, we set 
	\[ X_{(m,l),k,c}^{\delta,(\blacktriangle,\blacktriangledown)}(\C)=
		\left\{
			\begin{array}{c|c}
				(x(t),y(t))\in X_{(m,l),k,c}^{\delta}(\C) & 
				\begin{array}{l}
					(P-cQ)_{\gamma_{P-cQ}(\C)}(\ac{x(t)}, \ac{y(t)}) \blacktriangle 0 \\
					Q_{\gamma_{Q}(\C)}(\ac{x(t)}, \ac{y(t)}) \blacktriangledown 0
				\end{array}
			\end{array}
		\right\}.
	\]
\end{notation}
	With these notation, we have: 
	\begin{equation}\label{decompzetaomega}
	Z_{f-c,\omega,\C}^{\delta}(T)= 
	\sum_{(\blacktriangle,\blacktriangledown)\in \{=,\neq\}^2}Z_{f-c,\omega,\C}^{\delta,(\blacktriangle,\blacktriangledown)}(T).
        \end{equation}

	In sections \ref{sectionneqneq} and \ref{sectionnonneqneq} we study the rationality and the limit of each motivic
	zeta function $Z_{f-c,\omega,\C}^{\delta,(\blacktriangle,\blacktriangledown)}(T)$ with 
	$(\blacktriangle,\blacktriangledown) \in \{=,\neq\}^2$ .

\subsection{Rationality and limit of $Z_{f-c,\omega,\C}^{\delta,(\neq,\neq)}(T)$} \label{sectionneqneq}
Let $C$ be a cone of $\E_c$.
We start by introducing notation used through out the section, then we study in section \ref{Cdim2} the case where $\dim C=2$ and in section \ref{Cdim1} the case where $\dim C=1$.
\begin{notations} \label{notationsactionsformdiff} \label{actionGmGm2} \label{notation-identification-action} ~
		Let $(\alpha,\beta) \in C \cap (\NN_{\geq 1})^2$. By construction of the Grothendieck ring of varieties in 
		\cite[\S 2]{GuiLoeMer05a} and \cite[\S 2]{GuiLoeMer06a} (see \cite[Proposition 3.13]{Rai11}), the class 
		\[
			\left[(P-cQ)_{\gamma(P-cQ)}/Q_{\gamma(Q)}\colon\Gm^2 \setminus \{(P-cQ)_{\gamma(P-cQ)}=0\} \cup \{Q_{\gamma(Q)}=0\} \to \Gm, \sigma_{\alpha,\beta}\right]
		\]
		with $\sigma_{\alpha,\beta}$ the action of $\Gm$ defined by
		$\sigma_{\alpha,\beta}(\lambda,(x,y))=(\lambda^\alpha x,\lambda^\beta y)$, does not depend on $(\alpha,\beta)$ in 
		$C \cap (\NN_{\geq 1})^2$. We replace $\sigma_{\alpha,\beta}$ by $\sigma_{C}$.
\end{notations}
\begin{rem}\label{rem1} 
	In this section, we consider arcs $(x(t),y(t))$ such that 
	\[(P-cQ)_{\gamma_{P-cQ}(\C)}(\ac{x(t)}, \ac{y(t)}) \neq 0\:\:\text{and}\:\: Q_{\gamma_{Q}(\C)}(\ac{x(t)}, \ac{y(t)}) \neq 0\]
	then by Remark \ref{rem:n-inferieur-m(ord{x},ord{y})} we have 
	\[
		\left\{
			\begin{array}{lcl}
				\ord{(f-c)(x(t),y(t))} & = & (\m_{P-cQ}-\m_Q)(\Ord{x(t)}, \Ord{y(t)})\\
				\m_{P-cQ}(\Ord{x(t)}, \Ord{y(t)}) & > & \m_Q(\Ord{x(t)}, \Ord{y(t)})\\
				\m_Q(\Ord{x(t)}, \Ord{y(t)}) & \leq & (\m_{P-cQ}-\m_Q)(\Ord{x(t)}, \Ord{y(t)})\delta
			\end{array}
		\right. .
	\]
\end{rem}
\begin{notation} \label{notationsneqneq}
	Using Remark \ref{rem1}, we consider the cone 
	\[
		\C_{(\neq,\neq)}^{\delta} = \{(\alpha,\beta)\in \C \mid \m_{P-cQ}(\alpha,\beta)>\m_{Q}(\alpha,\beta),\ 
		\m_{Q}(\alpha,\beta)\leq (\m_{P-cQ}(\alpha,\beta) - \m_{Q}(\alpha,\beta)) \delta\}
	\]
	and for any integer $n\geq 1$
	\[
		\C_{(\neq,\neq),n}^{\delta} = 
		\{(\alpha,\beta) \in \C_{(\neq,\neq)}^{\delta} \mid n=\m_{P-cQ}(\alpha,\beta) - \m_{Q}(\alpha,\beta)\}.
	\] 
	For any $(\alpha,\beta) \in \C_{(\neq,\neq)}^{\delta}$, we define the arc spaces
	\[X_{(\alpha,\beta)}^{\neq,\neq}=
		\left\{
			\begin{array}{c|c}
				(x(t),y(t))\in \L(\A^2_{\k}) &
				\begin{array}{l}
					\ord{x(t)}=\alpha,\: \ord{y(t)}=\beta\\
					(P-cQ)_{\gamma_{P-cQ}(\C)}(\ac{x(t)},\ac{y(t)})\neq 0\\
					Q_{\gamma_{Q}(\C)}(\ac{x(t)},\ac{y(t)})\neq 0 
				\end{array}
			\end{array}
		\right\}.
	\]
\end{notation}
\subsubsection{Case of dimension 2 cone} \label{Cdim2}
\begin{rem}\label{remCdim2}
	If $\dim{\C}=2$ then $\dim{\gamma_{P-cQ}(\C)} = \dim{\gamma_{Q}(\C)}=0$, we write
	\begin{equation} \label{faceconedim2}
		\gamma_{P-cQ}(\C)=(a_0,b_0), \gamma_{Q}(\C)=(a_1,b_1)\:\:\text{and}\:\: \C = \RR_{>0}\omega_1 + \RR_{>0}\omega_2
	\end{equation}
	where $\omega_1 \in \NN^2 $ and $\omega_2 \in \NN^2$ are two primitive vectors.
\end{rem}
\begin{prop}\label{conedim2}
	The series $Z_{f-c,\omega,\C}^{\delta,(\neq,\neq)}(T)$ is rational and there is $\delta_0>0$, such that for $\delta \geq \delta_0$:\\
	\begin{itemize}
		\item if $\gamma_{P-cQ}(\C) = (a_0,0)$ and $\gamma_{Q}(\C) = (a_1,0)$, then
			\begin{equation} \label{neqneqaxe1}
				-\lim_{T \to \infty} Z_{f-c,\omega,\C}^{\delta,(\neq,\neq)}(T) = \eps_{(a_0,a_1)} [x^{a_0-a_1} \colon \Gm \to \Gm, \sigma_{\Gm}]
			\end{equation}
			with $\eps_{(a_0,a_1)} = 1$ if $a_0>a_1$, otherwise $\eps_{(a_0,a_1)}=0$,\\

		\item if $\gamma_{P-cQ}(\C) = (0,b_0)$ and $\gamma_{Q}(\C) = (0,b_1)$, then
			\begin{equation} \label{neqneqaxe2}
				-\lim_{T \to \infty} Z_{f-c,\omega,\C}^{\delta,(\neq,\neq)}(T) = \eps_{(b_0,b_1)} [y^{b_0-b_1} \colon \Gm \to \Gm, \sigma_{\Gm}]
			\end{equation}
			with $\eps_{(b_0,b_1)}= 1$ if $b_0>b_1$, otherwise $\eps_{(b_0,b_1)}=0$, \\

		\item otherwise
			\begin{equation} \label{neqneqnonaxebis}
				-\lim_{T \to \infty} Z_{f-c,\omega,\C}^{\delta,(\neq,\neq)}(T) = 
				- \eps_{\C} \left[(P-cQ)_{\gamma_{P-cQ}(\C)}/Q_{\gamma_{Q}(\C)}\colon \Gm^2 \to \Gm, \sigma_{\C}\right] 
			\end{equation}
			with $ \eps_{\C} = 1$ if $(\omega_1 \mid (a_0-a_1,b_0-b_1))>0$ and $(\omega_2 \mid (a_0-a_1,b_0-b_1))>0$, otherwise $\eps_{\C} = 0$.
	\end{itemize}
\end{prop}
\begin{proof}
	Using notation (\ref{faceconedim2}) in Remark \ref{remCdim2}, for any $(\alpha,\beta) \in \C$, we have 
	\[\m_{P-cQ}(\alpha,\beta)=\alpha a_0 + \beta b_0\:\: \text{and}\:\: \m_{Q}(\alpha,\beta)= \alpha a_1 + \beta b_1.\]
	Remark that for any integer $n\geq 1$, if it is not empty, the set 
		$\C_{(\neq,\neq),n}^{\delta} \cap (\NN_{\geq 1})^2$ is infinite if and only if $b_0 = b_1 = 0$ or $a_0=a_1=0$. Indeed, assuming $(a_0,a_1)\neq (0,0)$, for instance $a_1\neq 0$, we have  
	\begin{itemize}
		\item if $b_1 \neq 0$ then the inequality $\alpha a_1 + \beta b_1 \leq n \delta$ is satisfied for only finitely many $(\alpha, \beta) \in (\NN_{\geq 1})^2$.
		\item if $b_1 = 0$ and $b_0 \neq 0$ then, there are finitely many $(\alpha, \beta) \in (\NN_{\geq 1})^2$ such that $\alpha a_1 \leq n \delta$ and 
			$n=\alpha (a_0-a_1) + \beta b_0$. 
		\item if $b_1 = b_0 = 0$ then without condition on $\beta$, if $\C_{(\neq,\neq),n}^{\delta} \cap (\NN_{\geq 1})^2$ is not empty, it is infinite.
	\end{itemize}
	Then, to prove the rationality of $Z_{f-c,\omega,\C}^{\delta,(\neq,\neq)}(T)$, we have to study the case $(b_0,b_1)\neq (0,0)$ (and similarly $(a_0,a_1)\neq (0,0)$) and the case $(b_0,b_1)=(0,0)$ (and similarly $(a_0,a_1)=(0,0)$). The rationality result and the bound $\delta_0$ will be obtained using Lemma \ref{lemmedescones} and \cite[Proposition 3]{carai-antonio} based on summation of geometric series.

	Case A. We assume $(b_0,b_1)\neq (0,0)$. The proof is similar to $(a_0,a_1)\neq (0,0)$.\\
At least one element of $\gamma_{P-cQ}(\C) \cup \gamma_{Q}(\C)$  does not belong to the axis $\RR_{>0}(1,0)$.\\
In that case, by finiteness of 
$\C_{(\neq,\neq),n}^{\delta} \cap (\NN_{\geq 1})^2$ and additivity of the measure, we can write 
\[
	Z_{f-c,\omega,\C}^{\delta,(\neq,\neq)}(T) = 
	\sum_{n\geq 1} \big(\sum_{(\alpha,\beta) \in \C_{(\neq,\neq),n}^{\delta}\cap (\Z_{>0})^2}\LL^{-(\nu-1)\alpha}
	\mes\big(X_{(\alpha,\beta)}^{(\neq,\neq)}\big)\:\big)T^n.
\]
It follows from the definition of the motivic measure that 
\[
	\mes\left(X_{(\alpha,\beta)}^{(\neq,\neq)}\right)=\LL^{-\alpha-\beta}
	[(P-cQ)_{\gamma_{P-cQ}(\C)}/Q_{\gamma_{Q}(\C)}\colon\Gm^2 \to \Gm, \sigma_{\alpha,\beta}].
\]
Applying Remark \ref{notation-identification-action}, we have
\[
	Z_{f-c,\omega,\C}^{\delta,(\neq,\neq)}(T) = 
	\left[(P-cQ)_{\gamma_{P-cQ}(\C)}/Q_{\gamma_{Q}(\C)}\colon\Gm^2 \to \Gm, \sigma_{\C}\right] 
	S_{\C^{\delta}_{(\neq, \neq)}}(T)
\]
with 
\[
	S_{\C^{\delta}_{(\neq, \neq)}}(T) = \sum_{n\geq 1} \sum_{(\alpha,\beta) \in \C_{(\neq,\neq),n}^{\delta}}\LL^{-\nu \alpha - \beta}T^{n}.
\]
Using Lemma \ref{lemmedescones}, the formal series $S_{\C^{\delta}_{(\neq, \neq)}}(T)$ is rational and 
$\lim\limits_{T \to \infty} S_{\C^{\delta}_{(\neq, \neq)}}(T)=\chi_{c}\big(\C^{\delta}_{(\neq, \neq)}\big).$

In the following we study the cone $\C^{\delta}_{(\neq, \neq)}$ and compute its Euler characteristic.
By notation (\ref{faceconedim2}) in Remark \ref{remCdim2} and Proposition \ref{prop:dualfan-Newton-local}, 
$\C^{\delta}_{(\neq, \neq)}$ is the set of $k_1\omega_1 + k_2 \omega_2$ with $(k_1,k_2)\in (\RR_{>0})^2$ such that 
\begin{equation}\label{conditionD}
	\begin{cases}
		k_1(\omega_1 \mid D(\C))+k_2(\omega_2 \mid D(\C))>0 \\ \\

		k_1(\omega_1 \mid \gamma_{Q}(\C))+k_2(\omega_2 \mid \gamma_{Q}(\C))\leq 
		\delta[k_1(\omega_1 \mid D(\C))+k_2(\omega_2 \mid D(\C))].
	\end{cases}
\end{equation}
with $D(\C):=\gamma_{P-cQ}(\C)-\gamma_{Q}(\C)$.

$\bullet$ Assume $(\omega_1 \mid D(\C))>0$ and $(\omega_2 \mid D(\C))>0$. 
Then for 
\[\delta \geq \max \left( \frac{ (\omega_1 \mid \gamma_Q(\C)) }{(\omega_1 \mid D(\C))}, 
	\frac{(\omega_2 \mid \gamma_Q(\C))}{(\omega_2 \mid D(\C))}\right)
\]
we have $\C^{\delta}_{(\neq, \neq)} = \C \simeq (\RR_{>0})^2$ and $\chi_c\big(\C^{\delta}_{(\neq, \neq)}\big) = 1$.

$\bullet$ Assume $(\omega_1 \mid D(\C))=0$ and $(\omega_2 \mid D(\C))>0$.

Then for any $(k_1,k_2)\in (\RR_{>0})^2$, we have 
\[k_1.0 + k_2(\omega_2 \mid D(\C))>0\]
and by condition (\ref{conditionD}), we consider the inequality 
\begin{equation}\label{ineq10}
	k_1(\omega_1 \mid \gamma_Q(\C)) \leq k_2[ \delta(\omega_2 \mid D(\C))-(\omega_2 \mid \gamma_{Q}(\C))].
\end{equation}
\begin{itemize}
	\item  If $(\omega_1 \mid \gamma_{Q}(\C))=0$ then we also have $(\omega_1 \mid \gamma_{P-cQ}(\C))=0$. As $\omega_1$, $\gamma_{Q}(\C)$ and $\gamma_{P-cQ}(\C)$ are elements of $\NN^2$, then by orthogonality relations, either $\omega_1=(1,0)$ and 
		$\gamma_{Q}(\C)\in \RR(0,1)$, $\gamma_{P-cQ}(\C) \in \RR(0,1)$ either $\omega_1=(0,1)$ and $\gamma_{Q}(\C)\in \RR(1,0)$, $\gamma_{P-cQ}(\C) \in \RR(1,0)$, these two cases are incompatible with the assumptions $a_0 \neq 0$, $a_1 \neq 0$ and $(b_0,b_1)\neq (0,0)$.
	\item  If $(\omega_1 \mid \gamma_{Q}(\C))\neq 0$ then the inequality \ref{ineq10} is non trivial and we have the homeomorphism
		\[
			\C^{\delta}_{(\neq, \neq)} \simeq 
			\{(k_1,k_2)\in (\RR_{>0})^2 \mid 	
				k_1(\omega_1 \mid \gamma_Q(\C)) \leq k_2[ \delta(\omega_2 \mid D(\C))-(\omega_2 \mid \gamma_{Q}(\C))]
			\}
		\]
		with an Euler characteristic equal to 0 because its complement in $(\RR_{>0})^2$ is an open cone.
\end{itemize}

$\bullet$ The case $(\omega_1 \mid D(\C))> 0$ and $(\omega_2 \mid D(\C)) = 0$ is similar.

$\bullet$ Assume $(\omega_1 \mid D(\C))< 0$ and $(\omega_2 \mid D(\C))> 0$.

We have $(\omega_1\mid \gamma_{Q}(\C))\geq 0$ and $(\omega_2\mid \gamma_{Q}(\C))\geq 0$. 
For $\delta > \frac{(\omega_2\mid\gamma_{Q}(\C))}{(\omega_1\mid \gamma_{Q}(\C))}$ we have the inequality
\[\frac{(\omega_1\mid \gamma_{Q}(\C))-\delta(\omega_1 \mid D(\C))}
{\delta(\omega_2 \mid D(\C))-(\omega_2\mid\gamma_{Q}(\C))}>
-\frac{(\omega_1 \mid D(\C))}{(\omega_2 \mid D(\C))}\]
We deduce that $\C^{\delta}_{(\neq, \neq)}$ is homeomorphic to the cone of $(\RR_{>0})^2$ defined by the inequality
\[
	k_1[(\omega_1\mid \gamma_{Q}(\C))-\delta(\omega_1 \mid D(\C))]
	\leq k_2[\delta(\omega_2 \mid D(\C))-(\omega_2\mid\gamma_{Q}(\C))]
\]
whose Euler characteristic is equal to zero as above.

$\bullet$ The case $(\omega_1 \mid D(\C))> 0$ and $(\omega_2 \mid D(\C))< 0$ is similar.

$\bullet$ If $(\omega_1 \mid D(\C))\leq 0$ and $(\omega_2 \mid D(\C))\leq 0$ then 
$\C^{\delta}_{(\neq, \neq)} = \emptyset$.\medskip

Case B. We assume $b_0=b_1=0$. The proof is similar for $a_0=a_1=0$ working with the axis $x=0$.
The faces $\gamma_{P-cQ}(\C)$ and $\gamma_{Q}(\C)$ are the horizontal faces $\gamma_{h}(P-cQ)$ and $\gamma_h(Q)$ and are contained in the axis $y=0$. We have 
\[\gamma_{(P-cQ)}(\C)=(a_0,0),\: \gamma_{Q}(\C)=(a_1,0)\]
and 
\[\C = \RR_{>0}(0,1) + \RR_{>0}\omega = \{(\alpha,\beta) \in (\RR_{>0})^2 \mid \beta p > \alpha q\}\]
with $\omega = (p,q) \in \NN \times \NN_{\geq 1}$ a primitive vector.
Then, we have 
\[\C^{\delta}_{\neq, \neq} = \{ (\alpha,\beta)\in \C \mid a_0\alpha>a_1 \alpha,\: a_1\alpha\leq \delta(a_0-a_1)\alpha\}.\]

Thus, if $a_1 \geq a_0$ then the cone $\C^{\delta}_{\neq, \neq}$ is empty. We assume now $a_0 > a_1$. 

If $\delta \geq \frac{a_1}{a_0-a_1}>0$ then $\C^{\delta}_{\neq, \neq} = \C$ and we can write 
\[Z_{f-c,\omega,\C}^{\delta,(\neq,\neq)}(T) = \sum_{\alpha \geq 1} \LL^{-(\nu-1)\alpha} \mes(X_{\alpha, c})T^n\]
with 
\[
	\begin{array}{lcl}
		X_{\alpha, c}  & = & \{(x(t),y(t)) \in \L(\A^2_{\k}) \mid \ord{x(t)}=\alpha,\ (\ord{x(t)}, \ord{y(t)})\in \C \} \\
		& = &
		\left\{
			\begin{array}{l|l}
				(x(t),y(t)) \in \L(\A^2_{\k}) &
				\ord{x(t)}=\alpha,\ \ord{y(t)} > \left[\frac{q \Ord{x(t)}}{p}\right]+1
			\end{array}
		\right\}
	\end{array}
\]
The equation (\ref{neqneqaxe1}) follows now from \cite[Proposition 3]{carai-antonio}. 
\end{proof}

\subsubsection{Case of dimension 1 cone}\label{Cdim1}
\begin{rem}
If $\dim{\C}=1$ then $\C=\RR_{>0}(p,q)$ with $(p,q) \in (\NN_{\geq 1})^2$ a primitive vector.
\end{rem}

We consider $(a_0,b_0) \in \gamma_{P-cQ}(\C)$ and $(a_1,b_1)\in \gamma_{Q}(\C)$.

\begin{prop}\label{conedim1}
	The series $Z_{f-c,\omega,\C}^{\delta,(\neq,\neq)}(T)$ is rational and there is $\delta_0>0$, such that for $\delta \geq \delta_0$:
	\[
		-\lim_{T \to \infty} Z_{f-c,\omega,\C}^{\delta,(\neq,\neq)}(T) = 
		- \eps_{\C}\left[\frac{(P-cQ)_{\gamma_{P-cQ}(\C)}}{Q_{\gamma_{Q}(\C)}}\colon \Gm^2 \setminus 
		\{(P-cQ)_{\gamma_{P-cQ}(\C)}Q_{\gamma_{Q}(\C)}=0\} \to \Gm, \sigma_{\C}\right]
	\]
	with $\eps_{\C}=-1$ if $((a_0-a_1,b_0-b_1)\mid(p,q))> 0$ and otherwise $\eps_{\C}=0$.
\end{prop}
\begin{proof} We proceed as in the proof of Proposition \ref{conedim2}.
		As $\C=\RR_{>0}(p,q)$, we have 
			\[\C_{(\neq,\neq)}^{\delta} = 
			\left\{\begin{array}{c|l}
					k(p,q) & 
					\begin{array}{l}
						k( (p,q)\mid (a_0-a_1,b_0-b_1)) > 0 \\
						k ( (p,q) \mid (a_1,b_1) ) \leq \delta k ( (p,q)\mid (a_0-a_1,b_0-b_1))
					\end{array}
				\end{array}
			\right\}.
		\]
		If $((a_0-a_1,b_0-b_1)\mid(p,q))\leq 0$ then $\C_{(\neq,\neq)}^{\delta}=\emptyset$ and $Z_{f-c,\omega,\C}^{\delta,(\neq,\neq)}(T) = 0$. 

		Otherwise, for 	$\delta > \frac{((a_1,b_1)\mid(p,q))}{((a_0-a_1,b_0-b_1)\mid(p,q))}$ we have $\C_{(\neq,\neq)}^{\delta} = \C$ and 
		\[Z_{f-c,\omega,\C}^{\delta,(\neq,\neq)}(T) = 
		\sum_{k\geq 1}\LL^{-(\nu-1)kp}\mes\big(X_{(kp,kq)}^{\neq,\neq}\big) T^{k(p(a_0-a_1) + q(b_0-b_1))}.\]
		It follows from the definition of the motivic measure, and Notation \ref{notation-identification-action}, that 
		\[
			\mes\big(X_{(kp,kq)}^{\neq,\neq}\big)=\LL^{-k(p+q)}
			\left[\frac{(P-cQ)_{\gamma_{P-cQ}(\C)}}{Q_{\gamma_{Q}(\C)}} \colon 
				\Gm^2 \setminus \{(P-cQ)_{\gamma_{P-cQ}(\C)}Q_{\gamma_{Q}(\C)}=0\} \to \Gm, \sigma_{C}
			\right]
		\]
		and we conclude applying geometric summation and Lemma \ref{lemmedescones}.
	\end{proof}
	\subsection{Rationality and limit of $Z_{f-c,\omega,\C}^{\delta,(=,\neq)}(T)$, $Z_{f-c,\omega,\C}^{\delta,(\neq,=)}(T)$ and $Z_{f-c,\omega,\C}^{\delta,(=,=)}(T)$} \label{section-eqneq} \label{section-eqeq} \label{section-neqeq} \label{sectionnonneqneq}
In this section we consider $C \in \E_c$ and study first the motivic zeta function $Z_{f-c,\omega,\C}^{\delta,(=,\neq)}(T)$ in Proposition \ref{prop:rationalityZeqneq}, then similarly the motivic zeta functions $Z_{f-c,\omega,\C}^{\delta,(\neq,=)}(T)$ and $Z_{f-c,\omega,\C}^{\delta,(=,=)}(T)$ in Proposition \ref{zfautrescas}.

\begin{rem}\label{notationseqleq}
	We consider arcs $(x(t),y(t))$ such that 
	\[(P-cQ)_{\gamma_{P-cQ}(\C)}(\ac{x(t)}, \ac{y(t)}) = 0\:\text{and}\: Q_{\gamma_{Q}(\C)}(\ac{x(t)}, \ac{y(t)}) \neq 0.\] 
	Then 
	\begin{itemize}
		\item  $\C=\RR_{>0}(p,q)$ with $(p,q)\in (\NN_{\geq 1})^2$ a primitive vector. The face $\gamma_{P-cQ}(\C)$ has dimension $1$ and is supported by a line of equation $ap+bq=N_c$. 
		\item  There is a root $\mu \in \R_{\gamma_{P-cQ}(\C)} \setminus \R_{\gamma_{Q(\C)}}$ such that 
			$(\ac{y(t)})^p = \mu (\ac{x(t)})^q$ and we have the conditions
			\[
				\begin{cases} 
					\Ord{(f-c)(x(t),y(t))}=\Ord{(P-cQ)(x(t),y(t))}-\m_Q(\Ord{x(t)}, \Ord{y(t)})\\
					\Ord{(P-cQ)(x(t),y(t))}>\max\left( \m_{P-cQ}(\Ord{x(t)}, \Ord{y(t)}), \m_Q(\Ord{x(t)}, \Ord{y(t)}) \right)\\
					\m_Q(\Ord{x(t)}, \Ord{y(t)})\leq [\Ord{(P-cQ)(x(t),y(t))}-\m_Q(\Ord{x(t)}, \Ord{y(t)})]\delta
				\end{cases}.
			\]
	\end{itemize}
\end{rem}
\begin{rem}
We consider a point $(a_1,b_1) \in \gamma_Q(\C)$ and set $N:=a_1p+b_1q$.

By additivity of the measure, the series $Z_{f-c,\omega,\C}^{\delta,(=,\neq)}(T)$ has the following decomposition 
\[
	Z_{f-c,\omega,\C}^{\delta,(=,\neq)}(T) = \sum_{\mu \in R_{\gamma_{P-cQ}(\C)}\setminus R_{\gamma_{Q}(\C)}}
	\sum_{n \geq 1} 
	\big(
		\sum_{\underset{n=m-\m_{Q}(\alpha,\beta)}{((\alpha,\beta),m)\in \C_{(=,\neq)}^{\delta}\cap (\NN_{\geq 1})^3}}
		\LL^{-(\nu-1)\alpha}\mes \big(X_{(\alpha,\beta),m,\mu}^{(=,\neq)}\big)
	\big)
	T^{n}
\]
where we consider the cone 
\[
	\C_{(=,\neq)}^{\delta} = \{((\alpha,\beta),m)\in \C \times \RR_{>0} \mid 
	m>\max(\m_{P-cQ}(\alpha,\beta),\:\m_{Q}(\alpha,\beta)),\ \m_{Q}(\alpha,\beta)\leq (m - \m_{Q}(\alpha,\beta)) \delta\}
\]
and for any $((\alpha,\beta),m) \in (\NN_{\geq 1})^3$ and $\mu \in R_{\gamma_{P-cQ}(\C)}\setminus R_{\gamma_{Q}(\C)}$, we define the arc space
\[
	X_{(\alpha,\beta),m,\mu}^{(=,\neq)}=
	\left\{
		\begin{array}{c|l}
			(x(t),y(t))\in \L(\A^2_{\k}) & 
			\begin{array}{l}
				\Ord{x(t)}=\alpha, \Ord{y(t)}=\beta, \ac{y(t)}^p = \mu \ac{x(t)}^q \\
				Q_{\gamma_{Q}(\C)}(\Ac{x(t)},\Ac{y(t)})\neq 0\\ \Ord{(P-cQ)(x(t),y(t))}=m
			\end{array}
		\end{array}
	\right\}.
\]
\end{rem}

\begin{lem} \label{lem:mesXnmueqneq} 
	Let $((\alpha,\beta),m) \in \C_{(=,\neq)}^{\delta}\cap \NN^3$, 
	let $\mu \in R_{\gamma_{P-cQ}(\C)}\setminus R_{\gamma_{Q}(\C)}$ 
	and $\tsigma_{(p,q,\mu)}$ be the induced Newton transform (Definition \ref{defn:Newton-map-local}). 
	Writing $(\alpha,\beta)=(pk,qk)$, for some $k>0$, we have the equality
	\[
		\mes\big(X_{(\alpha,\beta),m,\mu}^{(=,\neq)}\big)=\LL^{-(p+q-1)k} \mes\big(Y_{(m,k)}^{{\tsigma_{(p,q,\mu)}}}\big)
	\]
	with 
	$Y_{(m,k)}^{{\tsigma_{(p,q,\mu)}}} = 
	\left\{ 
		\begin{array}{c|c}
			\left(v(t),w(t)\right) \in \L(\AA^2_k) &
			\begin{array}{l} 
				\Ord{v(t)}=k,\:\Ord{w(t)}>0\\ 
				\Ord{(P-cQ)_{\tsigma_{(p,q,\mu)}}(v(t),w(t))} = m\:\\ 
				\Ord{Q_{\tsigma_{(p,q,\mu)}}(v(t),w(t))} = \m_{Q}(pk,qk)=kN 
			\end{array}
		\end{array} 
	\right\}.
	$ 

If $m\leq \m_{P-cQ}(\alpha,\beta)$ then the set $Y_{(m,k)}^{{\tsigma_{(p,q,\mu)}}}$ is empty.
\end{lem} 
\begin{proof}
	The proof is similar to that of \cite[Lemma 3.3]{CassouVeys13} (see also \cite[Proposition 6]{carai-antonio}).
\end{proof}

\begin{prop}\label{prop:rationalityZeqneq}  
	With above notation, we have the decomposition 
	\begin{equation} \label{eqZdelta<} 
			Z_{f-c,\omega,\C}^{\delta,(=,\neq)}(T)  =  
\sum_{\mu \in R_{\gamma_{P-cQ}}(\C)\setminus R_{\gamma_{Q}(\C)}}  \big(Z_{f_{\tsigma_{(p,q,\mu)}},\omega_{p,q,\mu}}^{\delta}(T)\big)_{((0,0),c)}
	\end{equation}
	with $\omega_{p,q}(v,w)=v^{\nu p + q - 1}dv \wedge dw$. In particular, for $\delta$ large enough, we have
	\[
		-\lim_{T \ra \infty}	
		Z_{f-c,\omega,\C}^{\delta,(=,\neq)}(T)  =  
		\sum_{\mu \in R_{\gamma_{P-cQ}}(\C)\setminus R_{\gamma_{Q}(\C)}}  S_{f_{\tsigma_{(p,q,\mu)}, (0,0),c}} \in \mgg.
	\]
\end{prop}
	\begin{proof} 
		Let $(\alpha,\beta)\in \C$. There is $k>0$ such that 
		$(\alpha,\beta)=(pk,qk)$, so  
		$\m_Q(\alpha,\beta) = \m_Q(pk,qk) = k{N}.$
			The intersection $\C_{(=,\neq)}^{\delta} \cap (\NN_{\geq 1})^3$ is homeomorphic to $\overline{\C}_{(=,\neq)}^{\delta} \cap (\NN_{\geq 1})^2$, with
		\[\overline{\C}_{(=,\neq)}^{\delta} = 
			\left\{ 
				\begin{array}{c|l} 
					(k,m) \in \RR_{>0}^2 & 
					\begin{array}{l} 
						\max(\m_{P-cQ}(pk,qk),\:\m_{Q}(pk,qk))<m\\
						\m_{Q}(pk,qk)\leq (m - \m_{Q}(pk,qk)) \delta
					\end{array} 
				\end{array}
			\right\}.
		\]
For any integer $n$, we consider 
\(
	\overline{\C}_{(=,\neq),n}^{\delta}=\left\{k\in \RR_{>0} \mid (k,n+kN) \in \overline{\C}_{(=,\neq)}^{\delta}\right\}
\)
and
\[
	\C_{(=,\neq),n}^{\delta}=\left\{(\alpha,\beta)\in (\RR_{>0})^2 \mid 
	((\alpha,\beta),n+\m_{Q}(\alpha,\beta)) \in \C_{(=,\neq)}^{\delta}\right\}.
\]
By Remark \ref{notationseqleq} and Lemma \ref{lem:mesXnmueqneq}, we have 
\[ 
	Z_{f-c,\omega,\C}^{\delta,(=,\neq)}(T) 
	=  \suml_{\mu \in R_{\gamma_{P-cQ}(\C)}\setminus R_{\gamma_{Q}(\C)}} 
	\sum_{n\geq 1} 
	\big(
		\suml_{k \in \overline{\C}_{(=,\neq),n}^{\delta} \cap \NN^*}
		\LL^{-(\nu p+q-1)k} \mes\big(Y_{(n+kN,k)}^{\tsigma_{(p,q,\mu)}}\big)
	\big)
	T^{n}.
\]
For any root $\mu \in R_{\gamma_{P-cQ}(\C)}\setminus R_{\gamma_{Q}(\C)}$, by definition of the zeta function in formula (\ref{zeta1}), we have  
\[ 
	(Z_{f_{\tsigma_{(p,q,\mu)}},\omega_{p,q}}^{\delta}(T))_{((0,0),c)} =  
	\sum_{n \geq 1} \big(\sum_{k \geq 1} \LL^{-(\nu p+q-1)k}
	\mes\big(Y_{n,k,c}^{\delta,{\tsigma_{(p,q,\mu)}}}\big)\big)T^{n} 
\]
with $\omega_{p,q}=v^{\nu p +q-1}dv\wedge dw$ and for any $n \geq 1$, $k \geq 1$
\[
	Y_{n,k,c}^{\delta,{\tsigma_{(p,q,\mu)}}}  =   
	\left\{
		\begin{array}{c|c} 
			\left(v(t),w(t)\right) \in \L(\AA^2_k) &
			\begin{array}{l} 
				\Ord{v(t)} = k, \Ord{w(t)} >0\\ 
				\Ord{\omega_{p,q}(v(t),w(t))} = (\nu p+q-1)k \\
				\Ord{(f_{\tsigma_{(p,q,\mu)}}-c)(v(t),w(t))} = n \\
				\Ord{Q_{\tsigma_{(p,q,\mu)}}(v(t),w(t))}\leq n\delta
			\end{array}
		\end{array} 
	\right\}.
\]
but $Q_{\tsigma(p,q,\mu)}=v^N\unit(v,w)$ with $\unit(0,0)\neq 0$, and for any $(v(t),w(t))\in Y_{n,k,c}^{\delta,{\tsigma_{(p,q,\mu)}}}$ we have
\[\ord{(P-cQ)_{\tsigma(p,q,\mu)}(v(t),w(t))}>k N_{c}\:\:\text{namely}\:\:n+\ord{Q_{\tsigma(p,q,\mu)}(v(t),w(t))}>k N_c,\]
and we obtain (\ref{eqZdelta<}) thanks to  
\[
	(Z_{f_{\tsigma_{(p,q,\mu)}},\omega_{p,q}}^{\delta}(T))_{((0,0),c)} = \sum_{n\geq 1} 
	\sum_{k \in \overline{\C}_{(=,\neq),n}^{\delta} \cap \NN^*}
	\LL^{-(\nu p+q-1)k} \mes\big(Y_{(n+kN,k)}^{{\tsigma_{(p,q,\mu)}}}\big)T^{n}.
\]
\end{proof}

\begin{prop} \label{zfautrescas}
	With above notation, we have the decomposition 
	\begin{equation} \label{eqZdelta<bis} 
		Z_{f-c,\omega,\C}^{\delta,(\neq,=)}(T)  =  
		\sum_{\mu \in R_{\gamma_{Q}}(\C)\setminus R_{\gamma_{(P-cQ)}}(\C)}  
		\big(Z_{f_{\tsigma_{(p,q,\mu)}},\omega_{p,q,\mu}}^{\delta}(T)\big)_{((0,0),c)}
	\end{equation}
	\begin{equation} \label{eqZdelta<bisbis} 
		Z_{f-c,\omega,\C}^{\delta,(=,=)}(T)  =  
		\sum_{\mu \in R_{\gamma_{Q}}(\C)\cap R_{\gamma_{(P-cQ)}}(\C)}  
		\big(Z_{f_{\tsigma_{(p,q,\mu)}},\omega_{p,q,\mu}}^{\delta}(T)\big)_{((0,0),c)}
	\end{equation}
	with $\omega_{p,q}(v,w)=v^{\nu p + q - 1}dv \wedge dw$. In particular, for $\delta$ large enough, we have
	\[
		-\lim_{T \ra \infty}	
		Z_{f-c,\omega,\C}^{\delta,(\neq,=)}(T)  =  
		\sum_{\mu \in R_{\gamma_{Q}}(\C)\setminus R_{\gamma_{(P-cQ)}}(\C)}  S_{f_{\tsigma_{(p,q,\mu)}},(0,0),c} \in \mgg
	\]
	and 
	\[
		-\lim_{T \ra \infty}	
		Z_{f-c,\omega,\C}^{\delta,(=,=)}(T)  =  
		\sum_{\mu \in R_{\gamma_{Q}}(\C)\cap R_{\gamma_{(P-cQ)}}(\C)}  S_{f_{\tsigma_{(p,q,\mu)}},(0,0),c} \in \mgg.
	\]
\end{prop}
\begin{proof}
	The proof is similar to that of Proposition \ref{prop:rationalityZeqneq}. 
\end{proof}
\subsection{Base cases} 
\begin{rem}\label{rembasecases}
	Applying the Newton algorithm simultaneously for $P-cQ$ and $Q$,  we have to consider (by Lemma \ref{lem:Newton-alg} and Definition \ref{def:algo-Newton}) the following base cases 	
	\[x^My^m \unit(x,y) \ \text{or} \ x^M(y-\mu x^q+g(x,y))^mu(x,y)\]
	with $(M,m)\in \NN^2$, $q\in \NN^*$, $g(x,y)=\suml_{a+bq>q}c_{a,b}x^ay^b$ and $\unit\in \k[x,y]$ with $\unit(0,0)\neq 0$.
	
	We only need to consider the following cases (\ref{monmon}) and (\ref{dim1dim1}): 
	\begin{equation} \label{monmon}
		P-cQ = x^{M_1}y^{m_1}\unit_{1}(x,y)\:\:\text{and}\:\: Q=x^{M_2}y^{m_2}\unit_2(x,y)
	\end{equation}
	with $(M_1,M_2)\in \NN^2$, $(m_1,m_2)\in \NN^2$ and $\unit_1$, $\unit_2$ are units,
	\begin{equation} \label{dim1dim1}
		\begin{cases}
			(P-cQ)(x,y)= x^{M_1}(y- \mu x^{q}+g(x,y))^{m_1}\unit_1(x,y)\\
			Q(x,y)= x^{M_2}(y- \mu x^{q}+g(x,y))^{m_2}\unit_2(x,y)
		\end{cases}
	\end{equation}
	with $(M_1,M_2)\in \left( \NN_{\geq 1} \right)^2$ and $(m_1,m_2)\in \NN^2\setminus \{(0,0)\}$, $g$ as above and $\unit_1$, $\unit_2$ are units.
\end{rem}
	 
\begin{proof}
	With above notation:
	\begin{itemize}
		\item Assume 
			\[
				\begin{cases}
					(P-cQ)(x,y)= x^{M_1}(y- \mu_1 x^{q_1}+g_1(x,y))^{m_1}\unit_1(x,y)\\
					Q(x,y)= x^{M_2}(y- \mu_2 x^{q_2}+g_2(x,y))^{m_2}\unit_2(x,y)
				\end{cases}.
			\]
			\begin{itemize}
				\item  	If $(\mu_1,q_1) \neq (\mu_2,q_2)$, then applying Lemma \ref{lem:Newton-alg}, and Newton maps $\tsigma=\tsigma_{(1,q_1,\mu_1)}$ or ${\tsigma=\tsigma_{(1,q_2,\mu_2)}}$, we are in the case (\ref{dim1dim1}) because $(P-cQ)_{\tsigma}$ or $Q_{\tsigma}$, has the form $x^M\unit(x,y)$ with $\unit$ a unit and $M\geq 1$ and the other has the form \[x^{M'}(y- \mu x^{q'}+g'(x,y))^{m'}\unit'(x,y)\] with $M'\in \NN^*$, $q' \in \NN^*$, $g'$ as above and $\unit'$ a unit.

				\item  	If $(\mu_1,q_1) = (\mu_2, q_2)=:(\mu,q)$ and $g_1 \neq g_2$, then the branches 
					\[h_1(x,y) := y- \mu x^q+g_1(x,y)\: \text{and}\: h_2(x,y)=y- \mu x^q+g_2(x,y)\]
					are different. Applying the Newton algorithm, there exists a composition of Newton transforms $\tsigma$, such that
					\[h_{1, \tsigma} = x^{L_1}(y-\mu'_{1} x^{q'_1} + g'_{1}(x,y))\:\text{and}\:
					  h_{2, \tsigma} = x^{L_2}(y-\mu'_{2} x^{q'_2} + g'_{2}(x,y))
				        \]
					with $(\mu'_{1},q'_1)\neq (\mu'_{2},q'_2)$. Then, we are back to the previous case.

				\item  	If $(\mu_1,q_1)=(\mu_2,q_2)=:(\mu,q)$, $g_1=g_2$ and $M_1=M_2=0$, then applying the Newton map $\tsigma:=\sigma_{(1,q,\mu)}$ to $P-cQ$ and $Q$, we get $(P-cQ)_{\tsigma}$ and $Q_{\tsigma}$ of case (\ref{dim1dim1}).
			\end{itemize}
		\item Assume 
			\[
				\begin{cases}
					(P-cQ)(x,y) = x^{M_1}y^{m_1}\unit_1(x,y)\\
					Q(x,y)= x^{M_2}(y- \mu_2 x^{q_2}+g_2(x,y))^{m_2}\unit_2(x,y)
				\end{cases}.
			\]
			Applying the Newton map $\tsigma=\tsigma_{(1,q_2,\mu_2)}$, we have
			\[
				\begin{cases}
					Q_{\tsigma} = x^{M'_2}(y- \mu'_2 x^{q'_2}+g'_2(x,y))^{m_2}\unit'_2(x,y) \\
					(P-cQ)_{\tsigma}=x^{M'_1}u'(x,y)= x^{M'_1}(y- \mu'_2 x^{q'_2}+g'_2(x,y))^{0}\unit'_1(x,y)
				\end{cases}
			\]
			with $\unit'_1$ and $\unit'_2$ units, $g'_2$ as above, $M'_1, M'_2 \in \NN_{\geq 1}$. Then we are back to the case (\ref{dim1dim1}).\medskip

		\item  	Similarly we treat the case
			\[
				\begin{cases}
					(P-cQ)(x,y)= x^{M_1}(y- \mu_1 x^{q_1}+g_1(x,y))^{m_1}\\
					Q(x,y)= x^{M_2}y^{m_2}\unit(x,y)
				\end{cases}.
			\]
	\end{itemize}
\end{proof}

\begin{example} \label{casdebase1} We consider the case (\ref{monmon}).
	By Theorem \ref{thmSf00c} and Proposition \ref{conedim2}:

	If $(M_1,M_2)\in (\NN_{\geq 1})^2$ then
	\begin{itemize}
		\item if $M_1 = M_2$ then $S_{f,(0,0),c}=0$.
		\item if $M_1 \neq M_2$ and $m_1 = m_2 =0$ then 
			\[S_{f,(0,0),c} = \eps_{M_1,M_2}[x^{M_1-M_2}\colon\Gm \to \Gm,\sigma_{\Gm}]\]
			with $\eps_{(M_1,M_2)}=1$ if $M_1> M_2$ and otherwise 0.
		\item if $M_1 \neq M_2$ and $m_1 = m_2 \neq 0$ then $S_{f,(0,0),c}=0$. 
		\item if $M_1 \neq M_2$ and $m_1 \neq m_2$ then 
			\[S_{f,(0,0),c} = \eps_{(M_1,m_1),(M_2,m_2)}[x^{M_1-M_2}y^{m_1-m_2}\colon\Gm^2 \to \Gm,\sigma_{\Gm}]\]
			with $\eps_{(M_1,m_1),(M_2,m_2)}=-1$ if $M_1>M_2$ and $m_1>m_2$ and otherwise 0.
	\end{itemize}
	
	If $(m_1,m_2)\in (\NN_{\geq 1})^2$ then
	\begin{itemize}
		\item if $m_1 = m_2$ then $S_{f,(0,0),c}=0$.
		\item if $m_1 \neq m_2$ and $M_1 = M_2 =0$ then 
			\[S_{f,(0,0),c} = \eps_{m_1,m_2}[y^{m_1-m_2}\colon\Gm \to \Gm,\sigma_{\Gm}]\]
			with $\eps_{(m_1,m_2)}=1$ if $m_1> m_2$ and otherwise 0.
		\item if $m_1 \neq m_2$ and $M_1 = M_2 \neq 0$ then $S_{f,(0,0),c}=0$.
		\item if $M_1 \neq M_2$ and $m_1 \neq m_2$ then 
			\[S_{f,(0,0),c} = \eps_{(M_1,m_1),(M_2,m_2)}[x^{M_1-M_2}y^{m_1-m_2}\colon\Gm^2 \to \Gm,\sigma_{\Gm^2}]\]
			with $\eps_{(M_1,m_1),(M_2,m_2)}=-1$ if $M_1>M_2$ and $m_1>m_2$ and otherwise 0.
	\end{itemize}

	If $M_1=m_1=0$ then $P-cQ = \unit_1(x,y)$, so $(P-cQ)(0,0)\neq 0$ contradiction.

	If $M_2=m_2=0$ then $Q = \unit_2(x,y)$, so $Q(0,0)\neq 0$ contradiction.

	If $M_2=m_1=0$ (the case $M_2=m_1=0$ is similar), then the motivic zeta function is 
	\[
		\big(Z_{\f,X}^{\delta}(T)\big)_{((0,0),c)} = 
		\sum_{(\alpha ,\beta)\in C^{\delta}\cap \NN^2}\mes(X_{\alpha,\beta})T^{M\alpha-m\beta}
	\]
	with
	\[
		C^{\delta}  =  \{(\alpha,\beta)\in (\RR_{>0})^2 \mid M\alpha > m\beta ,\: m\beta \leq \delta (M\alpha-m\beta)\}
		=  \{(\alpha,\beta)\in (\RR_{>0})^2 \mid (1+1/\delta) m\beta \leq M\alpha\}
	\]
	then $\chi_{c}(C^{\delta})=0$ and we conclude that $S_{f,(0,0),c}=0$.
\end{example}

\begin{example} \label{casdebase2} 
	We consider the case (\ref{dim1dim1}). By Theorem \ref{thmSf00c} we have
	\begin{equation}\label{equationcasdebase}
		\begin{array}{ccl} 
			S_{f,(0,0),c} & = & \eps_{(M_1+m_1q, M_2 + m_2q)}[x^{(M_1-M_2)+q(m_1-m_2)}\colon\Gm \to \Gm, \sigma_{\Gm}] \\ \\
			& - & \eps_{\C}[x^{M_1-M_2}y^{m_1-m_2}\colon\Gm^2 \to \Gm, \sigma_{\Gm^2}] \\ \\
			& - & \eps_{D}[x^{M_1-M_2}(y-\mu x^q)^{m_1-m_2}\colon \Gm^2 \setminus (y=\mu x^q) \to \Gm, \sigma_{\Gm^2}] \\ \\
			& - & \eps_{(=,=)} [x^{M_1-M_2}\xi^{m_1-m_2} \colon \Gm^2 \to \Gm, \sigma_{\Gm^2}]
		\end{array}
	\end{equation}
	with $\C = \RR_{>0}(1,0) + \RR_{>0}(1,q)$, $D = \RR_{>0}(1,q)$ and 
	\begin{itemize}
		\item[$\bullet$] $\eps_{(M_1+m_1q, M_2 + m_2q)} = 1$ if and only if $M_1+m_1q> M_2 + m_2q$, and otherwise 0.
		\item[$\bullet$] $\eps_{\C} = 1$ if and only if 
			\[((M_1-M_2 , m_1-m_2) \mid (1,0))>0\:\text{and}\: ((M_1-M_2 , m_1-m_2) \mid (1,q))>0,\]
			and otherwise 0.
		\item[$\bullet$] $\eps_{D} = -1$ if and only if $((M_1-M_2 , m_1-m_2) \mid (1,q))>0$, and otherwise 0.
		\item[$\bullet$] $\eps_{(=,=)} = 1$ if and only if $m_1>m_2$ and $M_1+m_1q> M_2 + m_2q$, and otherwise 0.
	\end{itemize}
\end{example}
\begin{rem}\label{caracteristiqueEulercasdebase2} Two particular cases of Example \ref{casdebase2} which usually occur (as in Example \ref{exemple}):
	\begin{enumerate}
		\item \label{eulercasspecifique} if $\eps_{(M_1+m_1q, M_2 + m_2q)} = 1$, $\eps_{\C} = 1$, $\eps_{D} = -1$ and $\eps_{(=,=)} = 1$, 
			by \cite[Proposition 2.4]{carai-infini} and additivity of the Euler characteristic $\tilde{\chi_{c}}$, 
			we have $\tilde{\chi_{c}}\left(S_{f,(0,0),c}^{(1)}\right)=0$.
		\item \label{casspecifique} if $M_2\geq M_1$, $M_1+q>M_2$ and $(m_1,m_2)=(1,0)$, then 
			\begin{multline*}		
				S_{f,(0,0),c}  =  [x^{(M_1-M_2)+q}\colon\Gm \to \Gm, \sigma_{\Gm}] 
				- [x^{M_1-M_2}(y-\mu x^q)\colon \Gm^2 \setminus (y=\mu x^q) \to \Gm, \sigma_{\Gm^2}] \\
				-  [x^{M_1-M_2}\xi \colon \Gm^2 \to \Gm, \sigma_{\Gm^2}] = 0.
			\end{multline*}
			Indeed, using the change of variables $(x,y)\mapsto (x,\xi=y-\mu x^q)$, we get the equality 
			\[
				[x^{M_1-M_2}(y-\mu x^q)\colon \Gm^2 \setminus (y=\mu x^q) \to \Gm, \sigma_{\Gm^2}]
				=
				[x^{M_1-M_2}\xi\colon \Gm^2 \setminus (\xi=-\mu x^q) \to \Gm, \sigma_{\Gm^2}]
			\]
			and obtain that $S_{f,(0,0),c}=0$. 
	\end{enumerate}
\end{rem}
\begin{proof}
	By the proof of Theorem \ref{thmSf00c}, we get the three first terms of formula (\ref{equationcasdebase}) and it is enough to prove that
	\[
		-  \lim_{T \to \infty} Z^{(=,=),\delta}(T) = - \eps_{(=,=)} [x^{M_1-M_2}\xi^{m_1-m_2} \colon \Gm^2 \to \Gm, \sigma_{\Gm}]
	\]
	with  
	\[ 
		Z^{(=,=),\delta}(T) = \sum_{n\geq 1}
		\big(
			\sum_{\underset{n=l_1-l_2}{((k,kq),(l_1,l_2)) \in C^{\delta}_{(=,=)} \cap (\NN_{\geq 1})^4}} 
		        \LL^{-(\nu-1)k}\mes\big( X_{(k,kq),(l_1,l_2),\mu}^{(=,=)}\big)
		\big) T^{n}
	\]
	with 
	\[
		X_{(k,kq),(l_1,l_2),\mu}^{(=,=)}=
		\left\{
			\begin{array}{c|l}
				(x(t),y(t))\in \L(\A^2_{\k}) & 
				\begin{array}{l}
					\Ord{x(t)}=k, \Ord{y(t)}=kq\\
					\ac{y(t)} = \mu \ac{x(t)}^q \\
					\Ord{Q(x(t),y(t))}=l_2\
					\Ord{(P-cQ)(x(t),y(t))}=l_1 
					\end{array}
				\end{array}
			\right\}
		\]
	and 
	\[
		C^{\delta}_{(=,=)} = \{((k,kq),(l_1,l_2))\in \RR^{2} \times \RR^{2} \mid  l_1 > k(M_1+m_1q),\ l_2 > k(M_2 + m_2 q),\ l_1> l_2,\ l_2 \leq (l_1-l_2)\delta \}.
	\]
	We consider the polynomial $h(x,y) = y-\mu x^q + g(x,y)$ and $\tilde{l}=\Ord h(x(t),y(t))$. 
	We have 
	\[ 
		Z^{(=,=),\delta}(T) = \sum_{n\geq 1}
		\big(
			\sum_{\underset{n=k(M_1-M_2)+\tilde{l}(m_1-m_2)}{(k,\tilde{l}) \in \tilde{C}^{\delta} \cap (\NN_{\geq 1})^2}} 
			\LL^{-(\nu-1)k}\mes\big( X_{(k,kq),(kM_1+m_1 \tilde{l},kM_2+m_2 \tilde{l})}^{(=,=)}\big)\:
		\big) T^{n}
	\]
	with 
	\[
			\tilde{C}^{\delta} = 
			\left\{ 
				\begin{array}{c|l}
				     (k,\tilde{l}) \in (\RR_{>0})^2 & 
				     \begin{array}{l}
					     \tilde{l}>kq,  k(M_1-M_2)+\tilde{l}(m_1-m_2)>0,\\
					     kM_2+m_2 \tilde{l} \leq ( k(M_1-M_2)+\tilde{l}(m_1-m_2)) \delta
				     \end{array}
				 \end{array}
		        \right\}.
	\]
	As the polynomial $h$ is Newton non-degenerate 
	(see \cite{Gui02a, GuiLoeMer05a} and \cite[\S 3.3.3, Example 3]{carai-antonio} for details) 
	we have for any $(k,\tilde{l}) \in \tilde{C}^{\delta}$ 
		\begin{multline*} 
			\mes\left( X_{(k,kq),(kM_1+m_1 \tilde{l},kM_2+m_2 \tilde{l}),\mu}^{(=,=)}\right)  =  
			\LL^{-k-\tilde{l}}\left[x^{M_1-M_2}\xi^{m_1-m_2} \colon (y=\mu x^q)\cap \Gm^2 \times \Gm \to \Gm, \sigma_{\Gm}\right] \\
			 =  \LL^{-k-\tilde{l}}\left[x^{M_1-M_2}\xi^{m_1-m_2} \colon \Gm^2 \to \Gm, \sigma_{\Gm}\right].
		\end{multline*}
	Then, we study the rationality and the limit of the formal series 
	\[S_{\tilde{C}^{\delta}}(T) = \sum_{(k,\tilde{l}) \in \tilde{C}^{\delta}} \LL^{-\nu k - \tilde{l}}T^{k(M_1-M_2)+\tilde{l}(m_1-m_2)}.\]
        
	Assume $m_1>m_2$. Then for $\delta$ large enough we have $(m_1-m_2)\delta > m_2$.
	\begin{itemize}
		\item If $\frac{M_2-M_1}{m_1-m_2} < q$ then, 
			as 
			\[
				\frac{M_2 - \delta(M_1-M_2)}{\delta(m_1 - m_2) - m_2} \underset{\delta \to \infty}{\to}  
				\frac{M_2-M_1}{m_1-m_2}
			\]
		there is $\delta_0>0$ such that for any $\delta>\delta_0$ 
			$\frac{M_2 - \delta(M_1-M_2)}{\delta(m_1 - m_2) - m_2} \leq q.$

			For any $(k,\tilde{l})$ with $\tilde{l}>kq$, we have 
			\[\tilde{l}>kq>\frac{M_2-M_1}{m_1-m_2}k\:\text{and}\:
				\tilde{l}> kq > k \frac{M_2 - \delta(M_1-M_2)}{\delta(m_1 - m_2) - m_2}
			\]
			then we conclude that
			\[\tilde{C}^{\delta} = \{(k,\tilde{l}) \in (\RR_{>0})^2 \mid \tilde{l}>kq\}\]
			which is an open cone of dimension 2
			and by Lemma \ref{lemmedescones} we have
			\[\lim\limits_{T \to \infty} S_{\tilde{C}^{\delta}}(T) = \chi_{c}(\tilde{C}^{\delta}) = (-1)^2 = 1.\]
		\item If $\frac{M_2-M_1}{m_1-m_2} \geq q$, then for $\delta$ large enough, we have
			\[\frac{M_2 - \delta(M_1-M_2)}{\delta(m_1 - m_2) - m_2} >  \frac{M_2-M_1}{m_1-m_2} \geq q.\]
			Then, for any $(k,\tilde{l})\in (\NN_{\geq 1})^2$, if 
			$\tilde{l} \geq k \frac{M_2 - \delta(M_1-M_2)}{\delta(m_1 - m_2) - m_2}$ then, 
			$\tilde{l}>kq$ and \[k(M_1-M_2)+\tilde{l}(m_1-m_2)>0\]
			and we conclude that
			\[
				\tilde{C}^{\delta} = 
				\{(k,\tilde{l}) \in (\RR_{>0})^2 \mid kM_2+m_2 \tilde{l} \leq ( k(M_1-M_2)+\tilde{l}(m_1-m_2)) \delta\}
			\]
			and by Lemma \ref{lemmedescones} we have
			\(\lim\limits_{T \to \infty} S_{\tilde{C}^{\delta}}(T) = \chi_{c}(\tilde{C}^{\delta}) = 0.\)
	\end{itemize}

	Assume $m_1 = m_ 2 \neq 0$. We have 
	\[
		\tilde{C}^{\delta} = 
		\{(k,\tilde{l}) \in (\RR_{>0})^2 \mid \tilde{l}>kq,\ k(M_1-M_2)>0, kM_2+m_2 \tilde{l} \leq (k(M_1-M_2)) \delta \}.
	\]
	\begin{itemize}
		\item If $M_1>M_2$ then there is $\delta_0>0$ such that for any $\delta > \delta_0$, 
			we have $\frac{\delta(M_1-M_2)-M_2}{m_2} \geq q$ and 
			$$\tilde{C}^{\delta} = 
			\left\{ \begin{array}{c|l} (k,\tilde{l}) \in (\RR_{>0})^2 & kq <  \tilde{l} \leq k\frac{(M_1-M_2)\delta-M_2}{m_2} \end{array} 
			\right\}.$$ 
			By Lemma \ref{lemmedescones} we conclude
			$\lim\limits_{T \to \infty} S_{\tilde{C}^{\delta}}(T) = \chi_{c}(\tilde{C}^{\delta}) = 0.$
		\item If $M_1 \leq M_2$ then $\tilde{C}^{\delta} = \emptyset$ then $\lim_{T \to \infty} S_{\tilde{C}^{\delta}}(T) = 0$.
	\end{itemize}

	Assume $m_1 < m_2$. If $M_1 \leq M_2$ then $\tilde{C}^{\delta} = \emptyset$ then $\lim_{T \to \infty} S_{\tilde{C}^{\delta}}(T) = 0$.
	If $M_1>M_2$ then 
			\begin{itemize}
				\item if $q \geq \frac{M_1 - M_2}{m_2 - m_1}$ then for any $(k,\tilde{l}) \in \tilde{C}^{\delta}$ we have $\frac{\tilde{l}}{k}>q$ and $k(M_1-M_2)+\tilde{l}(m_1-m_2)>0$ implying the contradiction
					\[\frac{M_1 - M_2}{m_2 - m_1}>\frac{\tilde{l}}{k}>q.\]
					Then, $\tilde{C}^{\delta} = \emptyset$ and 
					$\lim\limits_{T \to \infty} S_{\tilde{C}^{\delta}}(T) = 0$.
				\item if $q <\frac{M_1 - M_2}{m_2 - m_1}$, for any $(k,\tilde{l})\in (\RR_{>0})^2$ the conditions 
					$kq<\tilde{l}$ and $k(M_1-M_2)+\tilde{l}(m_1-m_2)>0$ are equivalent to 
					$kq<\tilde{l}<k\frac{M_1-M_2}{m_2-m_1}$.
					As for any $\delta>0$, we have 
					\[\frac{\delta(M_1-M_2)-M_2}{m_2-\delta(m_1-m_2)}<\frac{M_1-M_2}{m_2-m_1},\]
					we conclude that 
					\[
						\tilde{C}^{\delta} = 
						\left\{
							\begin{array}{c|l} 
								(k,\tilde{l}) \in (\RR_{>0})^2 & kq <  \tilde{l} \leq k\frac{(M_1-M_2)\delta-M_2}{\delta(m_2-m_1)+m_2}
							\end{array}
						\right\}.
					\]
					By Lemma \ref{lemmedescones} we have
					$\lim\limits_{T \to \infty} S_{\tilde{C}^{\delta}}(T) = \chi_{c}(\tilde{C}^{\delta}) = 0.$
			\end{itemize}
\end{proof}

\section*{Acknowledgement}
The first author is partially supported by the spanish grants: PID2020-114750GB-C31 and PID2020-114750GB-C32. The second author was partially supported by the ANR grant ANR-15-CE40-0008 (Défigéo). The authors would like to thank the referee for his comments.

\bibliographystyle{abbrv}
\bibliography{biblio_complete.bib}

\begin{thebibliography}{10}

\bibitem{ArtCasLue05a}
E.~{Artal Bartolo}, P.~Cassou-Nogu{{\`e}}s, I.~Luengo, and
  A.~Melle-Hern\'andez.
\newblock {Quasi-ordinary power series and their zeta functions}.
\newblock {\em Mem. Amer. Math. Soc.}, 178(841):vi+85, 2005.

\bibitem{BodPic07a}
A.~Bodin and A.~Pichon.
\newblock Meromorphic functions, bifurcation sets and fibred links.
\newblock {\em Math. Res. Lett.}, 14(3):413--422, 2007.

\bibitem{BodPicSea09a}
A.~Bodin, A.~Pichon, and J.~Seade.
\newblock Milnor fibrations of meromorphic functions.
\newblock {\em J. Lond. Math. Soc. (2)}, 80(2):311--325, 2009.

\bibitem{carai-antonio}
P.~Cassou-Nogu\`es and M.~Raibaut.
\newblock Newton transformations and the motivic {M}ilnor fiber of a plane
  curve.
\newblock In {\em Singularities, algebraic geometry, commutative algebra, and
  related topics}, pages 145--189. Springer, Cham, 2018.

\bibitem{carai-infini}
P.~Cassou-Nogu\`es and M.~Raibaut.
\newblock Newton transformations and motivic invariants at infinity of plane
  curves.
\newblock {\em Math. Z.}, 299(1-2):591--669, 2021.

\bibitem{CassouVeys13}
P.~Cassou-Nogu\`es and W.~Veys.
\newblock Newton trees for ideals in two variables and applications.
\newblock {\em Proc. Lond. Math. Soc. (3)}, 108(4):869--910, 2014.

\bibitem{Cassou-Nogues-Veys-15}
P.~Cassou-Nogu\`es and W.~Veys.
\newblock The {N}ewton tree: geometric interpretation and applications to the
  motivic zeta function and the log canonical threshold.
\newblock {\em Math. Proc. Cambridge Philos. Soc.}, 159(3):481--515, 2015.

\bibitem{CNS}
A.~Chambert-Loir, J.~Nicaise, and J.~Sebag.
\newblock {\em Motivic integration}, volume 325 of {\em Progress in
  Mathematics}.
\newblock Birkh\"{a}user/Springer, New York, 2018.

\bibitem{Cox}
D.~A. Cox, J.~B. Little, and H.~K. Schenck.
\newblock {\em Toric varieties}, volume 124 of {\em Graduate Studies in
  Mathematics}.
\newblock American Mathematical Society, Providence, RI, 2011.

\bibitem{Delgado-Maugendre}
F.~Delgado and H.~Maugendre.
\newblock Special fibres and critical locus for a pencil of plane curve
  singularities.
\newblock {\em Compositio Mathematica}, 136(1):69–87, 2003.

\bibitem{DenLoe98b}
J.~Denef and F.~Loeser.
\newblock {Motivic {I}gusa zeta functions}.
\newblock {\em J. Algebraic Geom.}, 7(3):505--537, 1998.

\bibitem{DenLoe99a}
J.~Denef and F.~Loeser.
\newblock {Germs of arcs on singular algebraic varieties and motivic
  integration}.
\newblock {\em Invent. Math.}, 135(1):201--232, 1999.

\bibitem{DenLoe01b}
J.~Denef and F.~Loeser.
\newblock {Geometry on arc spaces of algebraic varieties}.
\newblock In {\em {European {C}ongress of {M}athematics, {V}ol. {I}
  ({B}arcelona, 2000)}}, volume 201 of {\em {Progr. Math.}}, pages 327--348.
  Birkh{\"a}user, Basel, 2001.

\bibitem{DenLoe02a}
J.~Denef and F.~Loeser.
\newblock {Lefschetz numbers of iterates of the monodromy and truncated arcs}.
\newblock {\em Topology}, 41(5):1031--1040, 2002.

\bibitem{Fulton-toric}
W.~Fulton.
\newblock {\em Introduction to toric varieties}, volume 131 of {\em Annals of
  Mathematics Studies}.
\newblock Princeton University Press, Princeton, NJ, 1993.
\newblock The William H. Roever Lectures in Geometry.

\bibitem{GL}
M.~Gonz{\'a}lez~Villa and A.~Lemahieu.
\newblock The monodromy conjecture for plane meromorphic germs.
\newblock {\em Bull. Lond. Math. Soc.}, 46(3):441--453, 2014.

\bibitem{GVLM}
M.~Gonz{\'a}lez~Villa, A.~Libgober, and L.~Maxim.
\newblock Motivic infinite cyclic covers.
\newblock {\em Adv. Math.}, 298:413--447, 2016.

\bibitem{Gui02a}
G.~Guibert.
\newblock {Espaces d'arcs et invariants d'{A}lexander}.
\newblock {\em Comment. Math. Helv.}, 77(4):783--820, 2002.

\bibitem{GuiLoeMer05a}
G.~Guibert, F.~Loeser, and M.~Merle.
\newblock {Nearby cycles and composition with a nondegenerate polynomial}.
\newblock {\em International Mathematics Research Notices},
  2005(31):1873--1888, 01 2005.

\bibitem{GuiLoeMer06a}
G.~Guibert, F.~Loeser, and M.~Merle.
\newblock {Iterated vanishing cycles, convolution, and a motivic analogue of a
  conjecture of {S}teenbrink}.
\newblock {\em Duke Math. J.}, 132(3):409--457, 2006.

\bibitem{GuiLoeMer09b}
G.~Guibert, F.~Loeser, and M.~Merle.
\newblock {Composition with a two variable function}.
\newblock {\em Math. Res. Lett.}, 16(3):439--448, 2009.

\bibitem{GusLueMel98a}
S.~M. Gusein-Zade, I.~Luengo, and A.~Melle-Hern\'andez.
\newblock {Zeta functions of germs of meromorphic functions, and the {N}ewton
  diagram}.
\newblock {\em Funktsional. Anal. i Prilozhen.}, 32(2):26--35, 95, 1998.

\bibitem{GusLueMel01a}
S.~M. Gusein-Zade, I.~Luengo, and A.~Melle-Hern\'andez.
\newblock {Bifurcations and topology of meromorphic germs}.
\newblock In {\em {New developments in singularity theory ({C}ambridge,
  2000)}}, volume~21 of {\em {NATO Sci. Ser. II Math. Phys. Chem.}}, pages
  279--304. Kluwer Acad. Publ., Dordrecht, 2001.

\bibitem{Kon95a}
M.~Kontsevich.
\newblock {String cohomology}.
\newblock {\em Lecture at Orsay}, D{\'e}cembre 7, 1995.

\bibitem{Kou76a}
A.~G. Kouchnirenko.
\newblock {Poly{\`e}dres de {N}ewton et nombres de {M}ilnor}.
\newblock {\em Invent. Math.}, 32(1):1--31, 1976.

\bibitem{LeRamanujam}
D.~T. L{\^e} and C.~P. Ramanujam.
\newblock The invariance of {Milnor}'s number implies the invariance of the
  topological type.
\newblock {\em Am. J. Math.}, 98:67--78, 1976.

\bibitem{TraWeb97a}
D.~T. L{\^e} and C.~Weber.
\newblock \'{E}quisingularit{\'e} dans les pinceaux de germes de courbes planes
  et {$C^0$}-suffisance.
\newblock {\em Enseign. Math. (2)}, 43(3-4):355--380, 1997.

\bibitem{Loo02a}
E.~Looijenga.
\newblock Motivic measures.
\newblock In {\em S\'eminaire Bourbaki : volume 1999/2000, expos\'es 865-879},
  number 276 in Ast\'erisque, pages 267--297. Soci\'et\'e math\'ematique de
  France, 2002.
\newblock talk:874.

\bibitem{Nguyen}
T.~T. Nguyen.
\newblock On the topology of rational functions in two complex variables.
\newblock {\em Acta Math. Vietnam.}, 37(2):171--187, 2012.

\bibitem{NST}
T.~T. Nguyen, T.~Saito, and K.~Takeuchi.
\newblock The bifurcation set of a rational function via {Newton} polytopes.
\newblock {\em Math. Z.}, 298(1-2):899--916, 2021.

\bibitem{Nguyen-Takeuchi}
T.~T. Nguyen and K.~Takeuchi.
\newblock Meromorphic nearby cycle functors and monodromies of meromorphic
  functions (with appendix by {T}. {S}aito).
\newblock {\em Rev. Mat. Complut.}, 36(2):663--705, 2023.

\bibitem{Par99a}
A.~Parusi{{\'n}}ski.
\newblock Topological triviality of {$\mu$}-constant deformations of type
  {$f(x)+tg(x)$}.
\newblock {\em Bull. London Math. Soc.}, 31(6):686--692, 1999.

\bibitem{Rai11}
M.~Raibaut.
\newblock {Singularit{\'e}s {\`a} l'infini et int{\'e}gration motivique}.
\newblock {\em Bull. Soc. Math. France}, 140(1):51--100, 2012.

\bibitem{Raibaut-fractions}
M.~Raibaut.
\newblock Motivic {M}ilnor fibers of a rational function.
\newblock {\em Rev. Mat. Complut.}, 26(2):705--734, 2013.

\bibitem{Sai90a}
M.~Saito.
\newblock {Mixed {H}odge modules}.
\newblock {\em Publ. Res. Inst. Math. Sci.}, 26(2):221--333, 1990.

\bibitem{SieTiba}
D.~Siersma and M.~Tib\u{a}r.
\newblock On the vanishing cycles of a meromorphic function on the complement
  of its poles.
\newblock In {\em Real and complex singularities}, volume 354 of {\em Contemp.
  Math.}, pages 277--289. Amer. Math. Soc., Providence, RI, 2004.

\bibitem{Tiba}
M.~Tib\u{a}r.
\newblock Singularities and topology of meromorphic functions.
\newblock In {\em Trends in singularities}, Trends Math., pages 223--246.
  Birkh\"auser, Basel, 2002.

\bibitem{Veys01}
W.~Veys.
\newblock Zeta functions and ``{K}ontsevich invariants'' on singular varieties.
\newblock {\em Canad. J. Math.}, 53(4):834--865, 2001.

\bibitem{Vey06a}
W.~Veys.
\newblock {Arc spaces, motivic integration and stringy invariants}.
\newblock In {\em {Singularity theory and its applications}}, volume~43 of {\em
  {Adv. Stud. Pure Math.}}, pages 529--572. Math. Soc. Japan, Tokyo, 2006.

\bibitem{VZG}
W.~Veys and W.~A. Z{\'u}{\~n}iga-Galindo.
\newblock Zeta functions and oscillatory integrals for meromorphic functions.
\newblock {\em Adv. Math.}, 311:295--337, 2017.

\bibitem{Wall}
C.~T.~C. Wall.
\newblock {\em Singular points of plane curves}, volume~63 of {\em Lond. Math.
  Soc. Stud. Texts}.
\newblock Cambridge: Cambridge University Press, 2004.

\end{thebibliography}
\end{document}